\documentclass[a4paper,11pt]{article}

\PassOptionsToPackage{quiet}{fontspec}
\usepackage{bm}
\usepackage[english]{babel}
\usepackage{hyperref}
\usepackage{amssymb,amsmath,amsthm}
\usepackage{enumerate}
\usepackage{enumitem}
\usepackage{graphicx, color}
\usepackage{epsfig}
\usepackage{tikz}
\usetikzlibrary{calc,decorations.pathmorphing,decorations.text}
\usepackage{subcaption}
\usepackage{refcount}
\usepackage[hmargin=2.5cm,vmargin=3cm]{geometry}
\usepackage{mathtools, braket}
\usepackage[normalem]{ulem}
\usepackage{authblk}
\usepackage{centernot}
\usepackage{indentfirst}
\usepackage{framed} 
\allowdisplaybreaks[4]
\usepackage[nameinlink]{cleveref}
\usepackage{amsmath}
\usepackage[numbers,sort&compress]{natbib}

\usepackage{hyperref}
\hypersetup{
	colorlinks=true,
	linkcolor=black,
	citecolor=black
}

\newtheorem{thm}{Theorem}[section]

\newtheorem{claim}{Claim}
\newtheorem{subclaim}{Subclaim}[claim]

\newtheorem{lem}[thm]{Lemma}

\newtheorem{proposition}[thm]{Proposition}

\newtheorem{problem}[thm]{Problem}

\newtheorem{observation}[thm]{Observation}

\renewcommand\ge\geqslant
\renewcommand\geq\geqslant
\renewcommand\le\leqslant
\renewcommand\leq\leqslant

\usepackage{caption}
\usetikzlibrary{matrix}
\usetikzlibrary{decorations.markings}
\tikzstyle{vertex}=[circle, draw, fill=black!50,
inner sep=0pt, minimum width=4pt]
\tikzset{->-/.style={decoration={
			markings,
			mark=at position .5 with {\arrow{>}}},postaction={decorate}}}

\tikzstyle{bigblue}=[color=blue, very thick, >=stealth]
\tikzstyle{lightblue}=[color=blue, thin, >=stealth]
\tikzstyle{bigred}=[color=red, very thick, >=stealth]
\tikzstyle{lightred}=[color=red, thin, >=stealth]
\tikzstyle{biggreen}=[color=black!30!green, very thick, >=stealth]
\tikzstyle{lightgreen}=[color=black!30!green,  thin, >=stealth]

\newlength{\bibitemsep}
\newlength{\bibparskip}
\let\oldthebibliography\thebibliography
\renewcommand\thebibliography[1]{
	\oldthebibliography{#1}
	\setlength{\parskip}{\bibitemsep}
	\setlength{\itemsep}{\bibparskip}
}

\title{Reconfiguration graphs for vertex colorings of $P_5$-free graphs}
\author{\small {Hui Lei$^1$, Yulai Ma$^2$, Zhengke Miao$^3$, Yongtang Shi$^4$, Susu Wang$^4$}\\
{\small $^1$ School of Statistics and Data Science, LPMC and KLMDASR}\\
{\small Nankai University, Tianjin 300071, China}\\
{\small $^2$ Department of Mathematics}\\
{\small Paderborn University, Warburger Str. 100, Paderborn 33098,
Germany}\\
{\small $^3$ School of Mathematics and Statistics} \\
{\small$ \& $  Key Laboratory of Analytical Mathematics and Applications (Ministry of Education)}\\
{\small Fujian Normal University, Fuzhou, Fujian 350007,  China}\\
{\small $^4$ Center for Combinatorics and LPMC}\\
{\small Nankai University, Tianjin 300071, China}\\
{\small Email: hlei@nankai.edu.cn; 
yulai.ma@upb.de;} \\ {\small zkmiao@jsnu.edu.cn; shi@nankai.edu.cn; susuwang@mail.nankai.edu.cn}\\
}
\date{\today}

\begin{document}
\maketitle
\begin{abstract}


For any positive integer $k$, the reconfiguration graph for all $k$-colorings of a graph $G$, denoted by $\mathcal{R}_k(G)$, is the graph where vertices represent the $k$-colorings of $G$, and two $k$-colorings are joined by an edge if they differ in color on exactly one vertex. Bonamy et al. established that for any $2$-chromatic $P_5$-free graph $G$, $\mathcal{R}_k(G)$ is connected for each $k\geq 3$. On the other hand, Feghali and Merkel proved the existence of a $7p$-chromatic $P_5$-free graph $G$  for every positive integer $p$, such that $\mathcal{R}_{8p}(G)$ is disconnected.  

In this paper, we offer a  detailed classification of the connectivity of $\mathcal{R} _k(G) $ concerning  $t$-chromatic $P_5$-free graphs $G$ for cases $t=3$, and $t\geq4$ with $t+1\leq k \leq {t\choose2}$. We demonstrate that $\mathcal{R}_k(G)$ remains connected for each $3$-chromatic $P_5$-free graph $G$ and each $k \geq 4$. Furthermore,  for each $t\geq4$ and $t+1 \leq k \leq {t\choose2}$,  we provide a construction of a $t$-chromatic $P_5$-free graph $G$ with $\mathcal{R}_k(G)$ being disconnected.  
This resolves a question posed by Feghali and Merkel.

\noindent\textbf{Keywords:} reconfiguration graphs; $P_5$-free graphs; frozen colorings; $k$-mixing\\
\end{abstract}

\section{Introduction}
\noindent Reconfiguration problems, spanning various fields, involve transforming solutions to a source problem into one another through elementary steps. 
These problems have been studied across various topics in graph theory, including vertex colorings, perfect matchings, independent sets, dominating sets, and more. For further details, readers are referred to surveys by Nishimura \cite{Nishimura2018} and van den Heuvel \cite{Van den2013}.

In this paper, we study reconfigurations for vertex colorings of graphs. All graphs under consideration are finite and simple. For undefined notation and terminology, we refer readers to \cite {Bondygraph2008}.  Let   $G=(V(G),E(G))$ be a graph, and $k$ be a positive integer. 
A  {\it proper $k$-coloring} of  $G$ is a mapping $\phi: V(G)\rightarrow \{1,2,\ldots,k\}$ such that
$\phi(u)\neq \phi(v)$ for any two adjacent vertices $u,v\in V(G)$. We simply write $k$-coloring for proper $k$-coloring in this paper, since all colorings under consideration are proper.
Additionally, $G$ is called {\it $k$-colorable} if it admits a proper $k$-coloring. The {\it chromatic number}  of $G$, denoted by $\chi(G)$,  is the smallest integer $k$ such that $G$ is
 $k$-colorable. In particular,  $G$  is called {\em $k$-chromatic} if $\chi(G)=k$.
 The {\it reconfiguration graph} for all $k$-colorings
of $G$, also called the {\it $k$-recoloring graph}, denoted by $\mathcal{R}_k(G)$, is the graph whose vertices are the $k$-colorings of $G$ and two
colorings are joined by an edge if they differ in color on exactly one vertex.  


As a major problem in this filed, the connectivity of $\mathcal{R}_k(G)$ has attracted widespread interest.
A result of Jerrum \cite{Jerrum1995} implies the existence of $k$ for each graph $G$, such that $\mathcal{R}_k(G)$ is connected. Precisely, he proved that $\mathcal{R}_k(G)$ is connected for each integer $k\geq \Delta(G)+2$. So it is of particular interest to investigate, for a given  class $\mathcal{G}$ of graphs, which values of $k$ make $\mathcal{R}_k(G)$ connected for each graph $G$ in $\mathcal{G}$. Notably,  it is observed that there exists no direct correlation between the connectivities of graphs $\mathcal{R}_i(G)$ and $\mathcal{R}_j(G)$ for a given graph $G$, where $i$ and $j$ are two integers with $i > j \geq \chi(G)$ (see Proposition \ref{RiRj}).
 Many related results have been proved in some special graphs classes, including degenerate graphs \cite{Bartiertreewidth22021,Bousquettrianglefree2022, Bousquetdelta2022, Bousquetsparse2016,Cambielist2024,FeghaliWeak2020}, planar graphs 
\cite{Bartiergirth52023, Chandranlist2022, Cranstonsparse2022,Cranston5colplanar2024,DvořákThomassen2021,Dvořákupdate2020,Eibenplanar2020, Feghalimad2021}, and perfect graphs \cite{Marthechordal2014, CerecedaConn2008, FeghaliWeak2020, FeghaliMix2022, Merkelweaklychordal2022}.  In this paper, we focus on $P_\ell$-free graphs, which contain no induced path of length $\ell-1$. 
 
Bonamy and Bousquet \cite{BonamyTree2018}  proved that, for each $t\geq1$ and $k\geq t+1$, $\mathcal{R}_{k}(G) $ is connected
for each $t$-chromatic $P_4$-free graph $G$. It is worth noting that the lower bound of $k$ is optimal, because any $t$-coloring of  $K_t$  is  an isolated vertex in  $\mathcal{R}_{t}(K_t) $, where $K_t$ is the complete graph of order $t$. 
A natural question arises: Does the analogue hold for $P_\ell$-free graphs for $\ell\geq5$? This property is trivially satisfied by $1$-chromatic $P_\ell$-free graphs. But, unfortunately, based on a result of Cereceda, van den Heuvel, and Johnson \cite{CerecedaConn2008}, Bonamy and Bousquet \cite{BonamyTree2018} observed that it is not the case for $2$-chromatic $P_6$-free graphs.

\begin{proposition}[\cite{BonamyTree2018,CerecedaConn2008}]\label{2p6}
For any $k\geq3$, there exists a $2$-chromatic $P_6$-free graph $G$ with $\mathcal{R}_{k}(G)$ being disconnected. 
\end{proposition}

Actually,  the above proposition can be simply generalized to $t$-chromatic $P_6$-free graphs for all $t\geq3$. This completes the classification of the connectivity of $ R _k(G) $ for  $t$-chromatic $P_6$-free graphs $G$.  A directed proof of the following proposition is presented  in Section \ref{preli}. 

\begin{proposition}\label{tp6}
For any $t\geq2$, $k\geq t+1$ and $\ell \geq 6$, there exists a $t$-chromatic $P_{\ell}$-free graph $G$ with $\mathcal{R}_{k}(G)$ being disconnected. 
\end{proposition}

Concerning the class of $P_5$-free graphs, some related results exist for certain subclasses \cite{Belavadihereditary2024, Belavadiforb2024, Belavadimodular2024, Marthechordal2014,  FeghaliWeak2020, FeghaliMix2022}, but few encompass the entire class of graphs. In \cite{Marthechordal2014}, Bonamy et al. proved the following theorem.

\begin{thm}[\cite{Marthechordal2014}]\label{2colorp5}
If $G$ is a $2$-chromatic $P_5$-free graph, then $\mathcal{R}_{k}(G) $ is connected for each $k\geq 3$.
\end{thm}

However, by a result of Feghali and Merkel \cite{FeghaliMix2022}, the analogue of Theorem \ref{2colorp5}  on the class of $t$-chromatic $P_5$-free graphs may not be true  for some values of $t$.

\begin{thm}[\cite{FeghaliMix2022}]\label{7col}
For every positive integer $p$, there exists a $7p$-chromatic $P_5$-free graph such that  $\mathcal{R}_{8p}(G) $ is disconnected.
 \end{thm}

In addition,  they also  proposed the  following open problem.
\begin{problem}[\cite{FeghaliMix2022}]\label{pro36}
For any $3\leq k\leq6$, whether $\mathcal{R}_{k+1}(G) $ is connected for any $k$-colorable $P_5$-free graph $G$.  
\end{problem}

We shall resolve Problem \ref{pro36}  in this paper. In fact, we prove a stronger result as follows, which offers a detailed classification of the connectivity of $ R _k(G) $ concerning  $t$-chromatic $P_5$-free graphs $G$ for cases $t=3$, and $t\geq4$ with $t+1\leq k \leq {t\choose2}$.

\begin{thm}\label{mainthm1}
If $G$ is a $3$-chromatic $P_5$-free graph, then $\mathcal{R}_{k}(G)$ is connected for each $k\geq4$.
 \end{thm}

 \begin{thm}\label{mainthm2}
 For any $t\geq4$ and  $t+1\leq k\leq\binom{t}{2}$, there exists a  $t$-chromatic $ P_5 $-free graph $G$ 
with $\mathcal{R}_{k}(G)$ being disconnected.
 \end{thm}

The organization of this paper is as follows.  Preliminaries are presented in the next section. In Sections \ref{main1} and \ref{main2}, we will prove Theorem~\ref{mainthm1} and  Theorem~\ref{mainthm2}, respectively. An open problem is proposed in a subsequent section.

\section{Preliminaries}\label{preli}

Given a graph $G=(V(G), E(G))$, we use $|V(G)|$ to denote the number of vertices,
$\delta(G)$  the minimum degree and  $diam(G)$ the diameter of $G$, respectively.  For any $x\in V(G)$, let $d_G(x)$  and $N_G(x)$ denote the degree   and neighborhood of $x$ in  $G$, respectively. Sometimes we omit the sign $G$ if there is no conflict occurs, such as using $d(x)$ and  $N(x)$ instead of $d_G(x)$ and $N_G(x)$. Let $N_{G}[x]=N_{G}(x)\cup\{x\}$. By convenience, we use $x\sim y$ to denote that $x$ is adjacent to $y$ and $x\not\sim y$ to denote that $x$ is not adjacent to $y$. A {\it clique} of $G$ is a set of mutually adjacent vertices, and that the maximum
size of a clique, the {\it clique number} of $G$, is denoted by $\omega(G)$. An ordered vertex pair $(x, y)$ is a {\it false twin} if $x\not\sim y$  and $N_G(x)\subseteq N_G(y)$.
For any two subsets $X$ and $Y$ of
$V(G)$, we denote by $[X, Y]_{G}$ the set of edges that has one end in $X$ and the other in $Y$. We say
that $X$ is {\it complete} to $Y$ or $[X, Y]_{G}$ is {\it complete} if every vertex in $X$ is adjacent to every vertex in $Y$. The subgraph of $G$ {\it induced} by $X$ is denoted by $G[X]$. For convenience, we simply write $G-X$ for $G[V(G)\setminus X]$. A set $X$ is called a {\it homogeneous\ set} if each vertex in $V(G)\setminus X$ is either complete to $X$ or has no neighbor in $X$. Two graphs $G$ and $H$ are {\it isomorphic}, denoted by $G\cong H$, if there are bijections $\theta:V(G)\rightarrow V(H)$ and $\phi:E(G)\rightarrow E(H)$ such that $\psi_G(e)=uv$ if and only if $\psi_H(\phi(e))=\theta(u)\theta(v)$.

Let  $C_n$  denote the cycle  on $n$ vertices.  
A graph is {\it $P_\ell$-free} if it does not contain 
$P_\ell$ as an induced graph, where $\ell$ is a positive integer.
A {\it gem} is a graph consisting of a vertex $v$ and an induced  path of lengh $3$ such that the vertex $v$ is complete to the vertex set of the path.  An {\it expansion} of a graph $H$ is any graph $G$ such that $V(G)$ can be partitioned into $|V(H)|$
nonempty sets $Q_v$, where $v\in V(H)$, such that $[Q_u,  Q_v]_{G}$ is complete if $uv\in E(H)$, and $[Q_u, Q_v]_{G}=\emptyset$ if
$uv\not\in E(H)$. Note that $H$ is an expansion of itself. An expansion of $H$ is a $P_4$-free
expansion if each $Q_v$ induces a $P_4$-free graph. 

For convenience,
$G$  is called {\it k-mixing} if $\mathcal{R}_k(G)$ is connected. For two $k$-colorings $\phi_1$ and $\phi_2$ of $G$, $\phi_1$ can be {\it transformed} to $\phi_2$ if there is a path from $\phi_1$ to $\phi_2$ in $\mathcal{R}_k(G)$. 
 A $k$-coloring of $G$ is {\it frozen} if it forms an isolated  vertex in $\mathcal{R}_k(G)$, in other words, if every vertex of $G$ whose closed neighborhood contains all $k$ colors, then it is a frozen $k$-coloring of $G$. 
Note that
the existence of a frozen $k$-coloring of a graph immediately implies that the graph is not
$k$-mixing.

Let $H_k$ be the complete bipartite graph on $2k$ vertices with two equal parts, and $M_k$ be a graph obtained from $H_k$ by removing a perfect matching. By construction of $M_k$, there exists a $k$-coloring  such that every color appears exactly once in each part of $M_k$. In addition, such a $k$-coloring is frozen.
The following proposition was proved in
\cite{CerecedaConn2008}.

\begin{proposition}[\cite{CerecedaConn2008}]\label{RiRj}
For  $k\geq 3$, the graph $M_k$ is a bipartite graph that is $i$-mixing for $3\leq i\leq k-1$ and $i\geq k+1$, but not $k$-mixing.
\end{proposition}

Note that $M_k$ is  $P_6$-free, $\chi(M_k)=\omega(M_k)=2$, and $M_k$ has a frozen $k$-coloring. So we can derive Proposition \ref{tp6} directly from Proposition \ref{p6free}.

\begin{proposition}\label{p6free}
Let $\ell\geq 3$, $t\geq 2$ and $k\geq t+1$. If $G$ is a $P_\ell$-free graph satisfying $\chi(G)=\omega(G)=t$ and $G$ has a frozen $k$-coloring, then for any $s\geq 1$, there exists a $P_\ell$-free graph $G'$ satisfying $\chi(G')=\omega(G')=t+s$ and $G$ has a frozen $(k+s)$-coloring. 
\end{proposition}

\begin{proof}
Let $G$ be a $P_\ell$-free graph satisfying $\chi(G)=\omega(G)=t$ and $G$ has a frozen $k$-coloring. Let $c$ be a frozen $k$-coloring of $G$. Without out loss of generality, we assume that $v_1$ is contained in a maximum clique and $c(v_1)=1$. Let $A=\{v_1,v_2,\ldots v_q\}$ denote the vertex set with color $1$ under $c$ and $B=V(G)\setminus A$. It is worth noting that $A$ is a dominating set of $G$. Let $G_1$ be the graph obtained from $G$ by adding a new independent vertex set $C=\{u_1,u_2,\ldots,u_q\}$ such that for any $i\in\{1,2,\ldots,q\}$, $u_i$ is adjacent to all vertices in $N_{G}[v_i]$.

Note that $G_1$ is a $(t+1)$-partite graph and $G_1[u_1\cup N_{G}[v_1]]$ contains a clique with size $t+1$. So $\chi(G_1)=t+1$. Now we give a $(k+1)$-coloring $c'$ of $G_1$ as follows: $c'(u)=k+1$ for $u\in C$  and $c'(v)=c(v)$ for any $v\in V(G)$. Since $c$ is a frozen $k$-coloring and $A$ is a dominating set of $G$, we know that $c'$ is a frozen $(k+1)$-coloring of $G_1$.  It remains to prove that $G_1$ is $P_{\ell}$-free. Suppose that $G_1$ contains an induced $P_{\ell}$. Let $P$ denote an induced $P_{\ell}$ of $G_1$. Since $G$ is  $P_\ell$-free,  $P$ contains  at least one  vertex of $C$. For any $i\in\{1,2,\ldots,q\}$, if $P$ contains $u_i$, then  $P$ contains no $v_i$ 
 otherwise $P$ is not an induced path as 
 $N_{G}(u_i)=N_{G}(v_i)$. Then $G$ contains an induced $P_{\ell}$ as we  replace all vertices in $P$ that belong to $C$ with the corresponding vertices in $A$, a contradiction. 

Note that $G_1$ is a $P_\ell$-free graph, satisfying that $\chi(G_1)=\omega(G_1)=t+1$ and $G_1$ has a frozen $(k+1)$-coloring. So we can perform the same operation on $G_1$ as mentioned above for $G$. Repeat the operation $s$ times, we get a $P_\ell$-free graph $G_s$, satisfying that $\chi(G_s)=\omega(G_s)=t+s$ and $G_s$ has a frozen $(k+s)$-coloring.
\end{proof}

\section{Proof of Theorem~\ref{mainthm1}}\label{main1}

In this section, for convenience of the induction hypothesis, we prove the following stronger version of Theorem~\ref{mainthm1}.

\begin{thm}\label{mainthm1.1}
If $G$ is a $3$-colorable $P_5$-free graph, then $\mathcal{R}_{k}(G)$ is connected for each $k\geq4$.
\end{thm}

\subsection{Prerequisites}

To complete the proof of Theorem~\ref{mainthm1.1}, we need a result  proved in \cite{FeghaliWeak2020} as follows.

\begin{thm}[\cite{FeghaliWeak2020}]\label{3p5}
If $G$ is a $3$-chromatic $(P_5,C_5,\overline{P_5})$-free graph, then $\mathcal{R}_{k}(G) $ is connected for each $k\geq 4$.
 \end{thm}

We also need the following useful fact.

\begin{observation}\label{K3C5}
The graph, obtained from $C_5$ by expanding each of two vertices to  $K_2$, has chromatic number at least $4$.
\end{observation}

\begin{proof}
If two adjacent vertices of $ C_5$  are each expanded to $K_2$, then the new graph contains a $K_4$, so it's  chromatic number is at least $4$. Suppose that $C_5=x_1x_2x_3x_4x_5x_1$ and  two nonadjacent vertices are each expanded to $K_2$. Without loss of generality we assume that $x_1$ and $x_3$ are expanded. Let $F$ denote the new graph with five nonempty sets $Q_{x_i}, i\in \{1,2,\ldots,5\}$. Let $y_i\in Q_{x_i}$ for $i\in \{1,2,\ldots,5\}$ and let $y_1\sim y'_1$, $y_3\sim y'_3$, where $y'_1\in Q_{x_1}$ and $y'_3\in Q_{x_3}$. Let $c$ be a  $k$-coloring  of $F$.
Since $y_4\sim y_5$,  we have that $c(y_2)\neq c(y_5)$ or $c(y_2)\neq c(y_4)$.
If $k\leq3$, then we have $c(y_1)= c(y'_1)$ when $c(y_2)\neq c(y_5)$ and  $c(y_3)= c(y'_3)$ when $c(y_2)\neq c(y_4)$, a contradiction. So $\chi(F)\geq 4$.
\end{proof}

Before presenting the formal proof of Theorem~\ref{mainthm1.1}, we provide a characterization of  ($P_5$, gem)-free graphs containing an induced $C_5$, as established in  \cite{ChudnovskyGem2020}.

\begin{lem}[\cite{ChudnovskyGem2020}]\label{Pgem}
Let $G$ be a connected ($P_5$,gem)-free graph that contains an induced $C_5$. Then either $G\in \mathcal{H}$ or $G$ is a $P_4$-free expansion of either $G_1$, $G_2$,$\ldots$, $G_9$ or $G_{10}$, where $G_1, G_2,\ldots, G_{10}$ are graphs shown in Figure \ref{basicgraphs}, and the graph class $\mathcal{H}$ defined as follows:  for any $H\in \mathcal{H}$, $H$ is a connected $(P_5,gem)$-free graph and  $V(H)$ can be partitioned into seven nonempty sets $A_1,A_2,\dots,A_7$ such that:

$\bullet$ Each $A_i$ induces a $P_4$-free graph.

$\bullet$ $[A_1,A_2\cup A_5\cup A_6]$ is complete and $[A_1,A_3\cup A_4\cup A_7]= \emptyset$.

$\bullet$ $[A_3,A_2\cup A_4\cup A_6]$ is complete and $[A_3,A_5\cup  A_7]= \emptyset$.

$\bullet$ $[A_4,A_5\cup A_6]$ is complete and $[A_4,A_2\cup  A_7]= \emptyset$.

$\bullet$ $[A_2,A_5\cup A_6\cup A_7]=\emptyset$ and $[A_5,A_6\cup  A_7]= \emptyset$.

$\bullet$ The vertex set of each component of $G[A_7]$ is a homogeneous set.

$\bullet$ Every vertex in $A_7$ has a neighbor in $A_6$.
\end{lem}

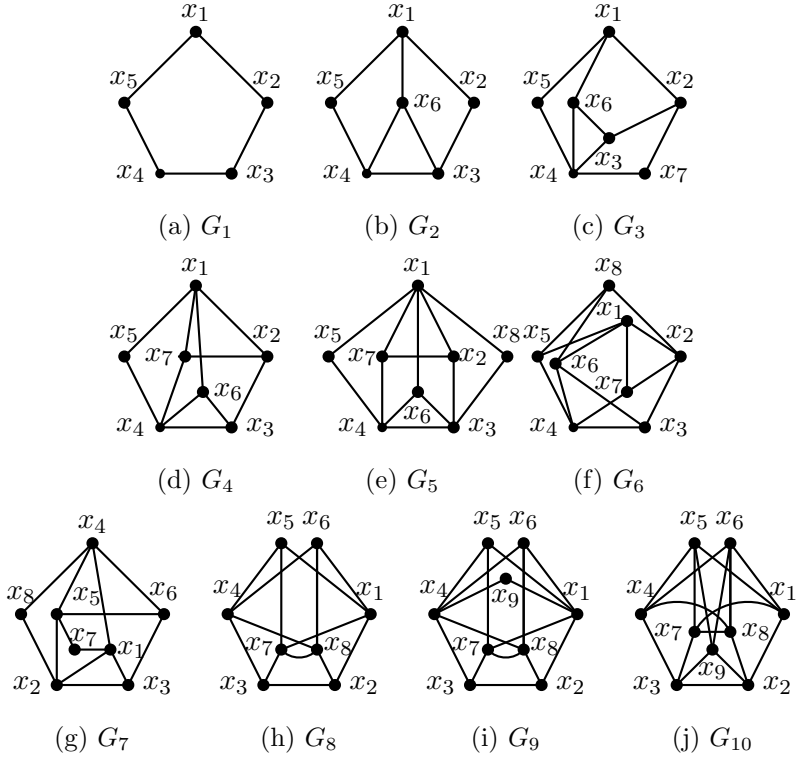
\begin{figure}[t]
	\centering
   \begin{subfigure}[t]{0.17\textwidth}
   	\centering
   	\begin{tikzpicture}[scale=0.47]
    	\draw [line width=0.8pt, black] (0,0) to (-1,2);
    	\draw [line width=0.8pt, black] (0,0) to (2,0);
    	\draw [line width=0.8pt, black] (-1,2) to (1,4);
    	\draw [line width=0.8pt, black] (1,4) to (3,2);
    	\draw [line width=0.8pt, black] (3,2) to (2,0);
    	
    	\draw [fill=black,line width=0.8pt] (0,0) node[left=0.5mm] {$x_4$} circle (3pt); 
    	\draw [fill=black,line width=0.8pt] (2,0) node[right=0.5mm] {$x_3$} circle (4pt); 
    	\draw [fill=black,line width=0.8pt] (3,2) node[above=0.5mm] {$x_2$} circle (4pt); 
    	\draw [fill=black,line width=0.8pt] (1,4) node[above] {$x_1$} circle (4pt); 
    	\draw [fill=black,line width=0.8pt] (-1,2) node[above=0.5mm] {$x_5$} circle (4pt); 
    \end{tikzpicture}
\caption{$G_1$}
\label{fig:G1}
\end{subfigure} 
\begin{subfigure}[t]{0.17\textwidth}
\centering
\begin{tikzpicture}[scale=.47]   
	\draw [line width=0.8pt, black] (0,0) to (-1,2);
	\draw [line width=0.8pt, black] (0,0) to (2,0);
	\draw [line width=0.8pt, black] (-1,2) to (1,4);
	\draw [line width=0.8pt, black] (1,4) to (3,2);
	\draw [line width=0.8pt, black] (3,2) to (2,0);
	\draw [line width=0.8pt, black] (1,4) to (1,2);
	\draw [line width=0.8pt, black] (1,2) to (0,0);
	\draw [line width=0.8pt, black] (1,2) to (2,0);
	
	\draw [fill=black,line width=0.8pt] (0,0) node[left=0.5mm] {$x_4$} circle (3pt); 
	\draw [fill=black,line width=0.8pt] (2,0) node[right=0.5mm] {$x_3$} circle (4pt); 
	\draw [fill=black,line width=0.8pt] (3,2) node[above=0.5mm] {$x_2$} circle (4pt); 
	\draw [fill=black,line width=0.8pt] (1,4) node[above] {$x_1$} circle (4pt); 
	\draw [fill=black,line width=0.8pt] (-1,2) node[above=0.5mm] {$x_5$} circle (4pt);
	\draw [fill=black,line width=0.8pt] (1,2) node[right=0.025mm] {$x_6$} circle (4pt);
\end{tikzpicture}
\caption{$G_2$}
\label{fig:G2}
\end{subfigure} 
\begin{subfigure}[t]{0.17\textwidth}
	\centering
	\begin{tikzpicture}[scale=0.47]   
		\draw [line width=0.8pt, black] (0,0) to (-1,2);
		\draw [line width=0.8pt, black] (0,0) to (2,0);
		\draw [line width=0.8pt, black] (-1,2) to (1,4);
		\draw [line width=0.8pt, black] (1,4) to (3,2);
		\draw [line width=0.8pt, black] (3,2) to (2,0);
		\draw [line width=0.8pt, black] (1,4) to (0,2);
		\draw [line width=0.8pt, black] (0,2) to (0,0);
		\draw [line width=0.8pt, black] (0,2) to (1,1);
		\draw [line width=0.8pt, black] (0,0) to (1,1);
		\draw [line width=0.8pt, black] (3,2) to (1,1);
		
		\draw [fill=black,line width=0.8pt] (0,0) node[left=0.5mm] {$x_4$} circle (3pt); 
		\draw [fill=black,line width=0.8pt] (2,0) node[right=0.5mm] {$x_7$} circle (4pt); 
		\draw [fill=black,line width=0.8pt] (3,2) node[above=0.5mm] {$x_2$} circle (4pt); 
		\draw [fill=black,line width=0.8pt] (1,4) node[above] {$x_1$} circle (4pt); 
		\draw [fill=black,line width=0.8pt] (-1,2) node[above=0.5mm] {$x_5$} circle (4pt);
		\draw [fill=black,line width=0.8pt] (0,2) node[right=0.025mm] {$x_6$} circle (4pt);
		\draw [fill=black,line width=0.8pt] (1,1) node[below=0.45mm] {$x_3$} circle (4pt);
	\end{tikzpicture}
	\caption{$G_3$}
	\label{fig:G3}
\end{subfigure}

\begin{subfigure}[t]{0.17\textwidth}
	\centering
	\begin{tikzpicture}[scale=0.47]   
		\draw [line width=0.8pt, black] (0,0) to (-1,2);
		\draw [line width=0.8pt, black] (0,0) to (2,0);
		\draw [line width=0.8pt, black] (-1,2) to (1,4);
		\draw [line width=0.8pt, black] (1,4) to (3,2);
		\draw [line width=0.8pt, black] (3,2) to (2,0);
		\draw [line width=0.8pt, black] (1,4) to (0.7,2);
		\draw [line width=0.8pt, black] (0.5,2) to (3,2);
		\draw [line width=0.8pt, black] (0,0) to (0.7,2);
		\draw [line width=0.8pt, black] (0,0) to (1.2,1);
		\draw [line width=0.8pt, black] (1,4) to (1.2,1);
		\draw [line width=0.8pt, black] (2,0) to (1.2,1);
		
		\draw [fill=black,line width=0.8pt] (0,0) node[left=0.5mm] {$x_4$} circle (3pt); 
		\draw [fill=black,line width=0.8pt] (2,0) node[right=0.5mm] {$x_3$} circle (4pt); 
		\draw [fill=black,line width=0.8pt] (3,2) node[above=0.5mm] {$x_2$} circle (4pt); 
		\draw [fill=black,line width=0.8pt] (1,4) node[above] {$x_1$} circle (4pt); 
		\draw [fill=black,line width=0.8pt] (-1,2) node[above=0.5mm] {$x_5$} circle (4pt);
		\draw [fill=black,line width=0.8pt] (0.7,2) node[left=0.0000000000001mm] {$x_7$} circle (4pt);
		\draw [fill=black,line width=0.8pt] (1.2,1) node[right=0.00000000000000000000001mm] {$x_6$} circle (4pt);
	\end{tikzpicture}
	\caption{$G_4$}
	\label{fig:G4}
\end{subfigure}
\begin{subfigure}[t]{0.17\textwidth}
	\centering
	\begin{tikzpicture}[scale=0.47]   
		\draw [line width=0.8pt, black] (0,0) to (-1.5,2);
		\draw [line width=0.8pt, black] (0,0) to (2,0);
		\draw [line width=0.8pt, black] (-1.5,2) to (1,4);
		\draw [line width=0.8pt, black] (1,4) to (3.5,2);
		\draw [line width=0.8pt, black] (3.5,2) to (2,0);
		\draw [line width=0.8pt, black] (2,0) to (2,2);
		\draw [line width=0.8pt, black] (0,2) to (0,0);
		\draw [line width=0.8pt, black] (0,2) to (2,2);
		\draw [line width=0.8pt, black] (0,0) to (1,1);
		\draw [line width=0.8pt, black] (1,4) to (0,2);
		\draw [line width=0.8pt, black] (1,4) to (2,2);
		\draw [line width=0.8pt, black] (1,1) to (2,0);
		\draw [line width=0.8pt, black] (1,1) to (1,4);
		
		\draw [fill=black,line width=0.8pt] (0,0) node[left=0.5mm] {$x_4$} circle (3pt); 
		\draw [fill=black,line width=0.8pt] (2,0) node[right=0.5mm] {$x_3$} circle (4pt); 
		\draw [fill=black,line width=0.8pt] (3.5,2) node[above=0.5mm] {$x_8$} circle (4pt); 
		\draw [fill=black,line width=0.8pt] (1,4) node[above] {$x_1$} circle (4pt); 
		\draw [fill=black,line width=0.8pt] (-1.5,2) node[above=0.5mm] {$x_5$} circle (4pt);
		\draw [fill=black,line width=0.8pt] (2,2) circle (4pt);
            \node at (2.55,2) {$x_2$};
		\draw [fill=black,line width=0.8pt] (0,2) circle (4pt);
             \node at (-0.55,2) {$x_7$};
		\draw [fill=black,line width=0.8pt] (1,1) node[below=0.45mm] {$x_6$} circle (4pt);
	\end{tikzpicture}
	\caption{$G_5$}
	\label{fig:G5}
\end{subfigure}
\begin{subfigure}[t]{0.17\textwidth}
	\centering
	\begin{tikzpicture}[scale=0.47]   
		\draw [line width=0.8pt, black] (0,0) to (-1,2);
		\draw [line width=0.8pt, black] (0,0) to (2,0);
		\draw [line width=0.8pt, black] (-1,2) to (1,4);
		\draw [line width=0.8pt, black] (1,4) to (3,2);
		\draw [line width=0.8pt, black] (3,2) to (2,0);
		\draw [line width=0.8pt, black] (0,0) to (-0.5,1.8);
		\draw [line width=0.8pt, black] (-0.5,1.8) to (1,4);
		\draw [line width=0.8pt, black] (-0.5,1.8) to (1.5,3);
		\draw [line width=0.8pt, black] (-0.5,1.8) to (2,0);
		\draw [line width=0.8pt, black] (1.5,3) to (3,2);
		\draw [line width=0.8pt, black] (1.5,1) to (3,2);
		\draw [line width=0.8pt, black] (1.5,1) to (1.5,3);
		\draw [line width=0.8pt, black] (1.5,1) to (0,0);
		\draw [line width=0.8pt, black] (-1,2) to (1.5,3);
		
		\draw [fill=black,line width=0.8pt] (0,0) node[left=0.5mm] {$x_4$} circle (3pt); 
		\draw [fill=black,line width=0.8pt] (2,0) node[right=0.5mm] {$x_3$} circle (4pt); 
		\draw [fill=black,line width=0.8pt] (3,2) node[above=0.5mm] {$x_2$} circle (4pt); 
		\draw [fill=black,line width=0.8pt] (1,4) node[above] {$x_8$} circle (4pt); 
		\draw [fill=black,line width=0.8pt] (-1,2) node[above=0.5mm] {$x_5$} circle (4pt);
		\draw [fill=black,line width=0.8pt] (-0.5,1.8) node[right=0.6mm] {$x_6$} circle (4pt);
		\draw [fill=black,line width=0.8pt] (1.5,1) circle (4pt);
            \node at (1,1.2) {$x_7$};
		\draw [fill=black,line width=0.8pt] (1.5,3) circle (4pt);
            \node at (1.05,3.3) {$x_1$};
	\end{tikzpicture}
	\caption{$G_6$}
	\label{fig:G6}
\end{subfigure}

\begin{subfigure}[t]{0.17\textwidth}
	\centering
	\begin{tikzpicture}[scale=0.47]   
		\draw [line width=0.8pt, black] (0,0) to (-1,2);
		\draw [line width=0.8pt, black] (0,0) to (2,0);
		\draw [line width=0.8pt, black] (-1,2) to (1,4);
		\draw [line width=0.8pt, black] (1,4) to (3,2);
		\draw [line width=0.8pt, black] (3,2) to (2,0);
		\draw [line width=0.8pt, black] (0,0) to (0,2);
		\draw [line width=0.8pt, black] (0,2) to (1,4);
		\draw [line width=0.8pt, black] (0,2) to (3,2);
		\draw [line width=0.8pt, black] (0,2) to (0.5,1);
		\draw [line width=0.8pt, black] (1.5,1) to (0.5,1);
		\draw [line width=0.8pt, black] (1.5,1) to (2,0);
		\draw [line width=0.8pt, black] (1.5,1) to (1,4);
		\draw [line width=0.8pt, black] (1.5,1) to (0,0);
				
		\draw [fill=black,line width=0.8pt] (1,4) node[above]  {$x_4$} circle (4pt); 
		\draw [fill=black,line width=0.8pt] (2,0) node[right=0.5mm] {$x_3$} circle (4pt); 
		\draw [fill=black,line width=0.8pt] (3,2) node[above=0.5mm] {$x_6$} circle (4pt); 
		\draw [fill=black,line width=0.8pt] (0,0) node[left=0.5mm] {$x_2$} circle (4pt); 
		\draw [fill=black,line width=0.8pt] (-1,2) node[above=0.05mm] {$x_8$} circle (4pt);
		\draw [fill=black,line width=0.8pt] (0,2) circle (4pt);
           \node at (0.85,2.5) {$x_5$};
		\draw [fill=black,line width=0.8pt] (1.5,1) circle (4pt);
           \node at (2.1,1) {$x_1$};
		\draw [fill=black,line width=0.8pt] (0.5,1) circle (4pt);
           \node at (0.75,1.35) {$x_7$};
	\end{tikzpicture}
	\caption{$G_7$}
	\label{fig:G7}
\end{subfigure}
\begin{subfigure}[t]{0.17\textwidth}
	\centering
	\begin{tikzpicture}[scale=0.47]   
		\draw [line width=0.8pt, black] (0,0) to (-1,2);
		\draw [line width=0.8pt, black] (0,0) to (2,0);
		\draw [line width=0.8pt, black] (-1,2) to (0.5,4);
		\draw [line width=0.8pt, black] (3,2) to (1.5,4);
		\draw [line width=0.8pt, black] (3,2) to (2,0);
		\draw [line width=0.8pt, black] (0.5,4) to (0.5,1);
		\draw [line width=0.8pt, black] (1.5,4) to (1.5,1);
		\draw [line width=0.8pt, black] (0,0) to (0.5,1);
		\draw [bend left=50, line width=0.8pt, black] (1.5,1) to (0.5,1);
		\draw [line width=0.8pt, black] (3,2) to (0.5,1);
		\draw [line width=0.8pt, black] (1.5,1) to (2,0);
		\draw [line width=0.8pt, black] (1.5,1) to (1.5,4);
		\draw [line width=0.8pt, black] (1.5,1) to (-1,2);
		\draw [line width=0.8pt, black] (-1,2) to (1.5,4);
		\draw [line width=0.8pt, black] (0.5,4) to (3,2);
		
		\draw [fill=black,line width=0.8pt] (0.5,4) node[above=0.5mm] {$x_5$} circle (4pt); 
    	\draw [fill=black,line width=0.8pt] (1.5,4) node[above=0.5mm] {$x_6$} circle (4pt); 
		\draw [fill=black,line width=0.8pt] (2,0) node[right=0.5mm] {$x_2$} circle (4pt); 
		\draw [fill=black,line width=0.8pt] (3,2) node[above=0.5mm] {$x_1$} circle (4pt); 
		\draw [fill=black,line width=0.8pt] (0,0) node[left=0.5mm] {$x_3$} circle (4pt); 
		\draw [fill=black,line width=0.8pt] (-1,2) node[above=0.8mm] {$x_4$} circle (4pt);
		\draw [fill=black,line width=0.8pt] (0.5,1) circle (4pt);
        \node at (-0.11,1.1) {$x_7$};
		\draw [fill=black,line width=0.8pt] (1.5,1) circle (4pt);
            \node at (2.1,1.1) {$x_8$};
		
	\end{tikzpicture}
	\caption{$G_8$}
	\label{fig:G8}
\end{subfigure}
\begin{subfigure}[t]{0.17\textwidth}
	\centering
	\begin{tikzpicture}[scale=0.47]   
		\draw [line width=0.8pt, black] (0,0) to (-1,2);
		\draw [line width=0.8pt, black] (0,0) to (2,0);
		\draw [line width=0.8pt, black] (-1,2) to (0.5,4);
		\draw [line width=0.8pt, black] (3,2) to (1.5,4);
		\draw [line width=0.8pt, black] (3,2) to (2,0);
		\draw [line width=0.8pt, black] (0.5,4) to (0.5,1);
		\draw [line width=0.8pt, black] (1.5,4) to (1.5,1);
		\draw [line width=0.8pt, black] (0,0) to (0.5,1);
		\draw [bend left=50, line width=0.8pt, black] (1.5,1) to (0.5,1);
		\draw [line width=0.8pt, black] (3,2) to (0.5,1);
		\draw [line width=0.8pt, black] (1.5,1) to (2,0);
		\draw [line width=0.8pt, black] (1.5,1) to (1.5,4);
		\draw [line width=0.8pt, black] (1.5,1) to (-1,2);
		\draw [line width=0.8pt, black] (-1,2) to (1.5,4);
		\draw [line width=0.8pt, black] (0.5,4) to (3,2);
		\draw [line width=0.8pt, black] (-1,2) to (1,3);
		\draw [line width=0.8pt, black] (1,3) to (3,2);
		
		\draw [fill=black,line width=0.8pt] (0.5,4) node[above=0.5mm] {$x_5$} circle (4pt); 
		\draw [fill=black,line width=0.8pt] (1.5,4) node[above=0.5mm] {$x_6$} circle (4pt); 
		\draw [fill=black,line width=0.8pt] (2,0) node[right=0.5mm] {$x_2$} circle (4pt); 
		\draw [fill=black,line width=0.8pt] (3,2) node[above=0.5mm] {$x_1$} circle (4pt); 
		\draw [fill=black,line width=0.8pt] (0,0) node[left=0.5mm] {$x_3$} circle (4pt); 
		\draw [fill=black,line width=0.8pt] (-1,2) node[above=0.8mm] {$x_4$} circle (4pt);
		\draw [fill=black,line width=0.8pt] (0.5,1) circle (4pt);
            \node at (-0.1,1.1) {$x_7$};
		\draw [fill=black,line width=0.8pt] (1.5,1)  circle (4pt);
            \node at (2.1,1.1) {$x_8$};
		\draw [fill=black,line width=0.8pt] (1,3) node[below=0.025mm] {$x_9$} circle (4pt);
		
	\end{tikzpicture}
	\caption{$G_9$}
	\label{fig:G9}
\end{subfigure}
\begin{subfigure}[t]{0.17\textwidth}
	\centering
	\begin{tikzpicture}[scale=0.47]   
		\draw [line width=0.8pt, black] (0,0) to (-1,2);
		\draw [line width=0.8pt, black] (0,0) to (2,0);
		\draw [line width=0.8pt, black] (-1,2) to (0.5,4);
		\draw [line width=0.8pt, black] (3,2) to (1.5,4);
		\draw [line width=0.8pt, black] (3,2) to (2,0);
		\draw [line width=0.8pt, black] (0.5,4) to (0.5,1.5);
		\draw [line width=0.8pt, black] (1.5,4) to (1.5,1.5);
		\draw [line width=0.8pt, black] (0,0) to (0.5,1.5);
		\draw [line width=0.8pt, black] (1.5,1.5) to (0.5,1.5);
		\draw [bend right=50, line width=0.8pt, black] (3,2) to (0.5,1.5);
		\draw [line width=0.8pt, black] (1.5,1.5) to (2,0);
		\draw [line width=0.8pt, black] (1.5,1.5) to (1.5,4);
		\draw [bend right=50, line width=0.8pt, black] (1.5,1.5) to (-1,2);
		\draw [line width=0.8pt, black] (-1,2) to (1.5,4);
		\draw [line width=0.8pt, black] (0.5,4) to (3,2);
		\draw [line width=0.8pt, black] (0.5,4) to (1,1);
		\draw [line width=0.8pt, black] (1,1) to (1.5,4);
		\draw [line width=0.8pt, black] (1,1) to (0,0);
		\draw [line width=0.8pt, black] (1,1) to (2,0);
		
		\draw [fill=black,line width=0.8pt] (0.5,4) node[above=0.5mm] {$x_5$} circle (4pt); 
		\draw [fill=black,line width=0.8pt] (1.5,4) node[above=0.5mm] {$x_6$} circle (4pt); 
		\draw [fill=black,line width=0.8pt] (2,0) node[right=0.5mm] {$x_2$} circle (4pt); 
		\draw [fill=black,line width=0.8pt] (3,2) node[above=0.5mm] {$x_1$} circle (4pt); 
		\draw [fill=black,line width=0.8pt] (0,0) node[left=0.5mm] {$x_3$} circle (4pt); 
		\draw [fill=black,line width=0.8pt] (-1,2) node[above=0.8mm] {$x_4$} circle (4pt);
		\draw [fill=black,line width=0.8pt] (0.5,1.5) node[left=0.025mm] {$x_7$} circle (4pt);
		\draw [fill=black,line width=0.8pt] (1.5,1.5) node[right=0.025mm] {$x_8$} circle (4pt);
		\draw [fill=black,line width=0.8pt] (1,1)  circle (4pt);
            \node at (1,0.4) {$x_9$};
		
	\end{tikzpicture}
	\caption{$G_{10}$}
	\label{fig:G10}
\end{subfigure}

\caption{\small Basic graphs.}
\label{basicgraphs}
\end{figure}

Based on Lemma \ref{Pgem}, we present some structural properties of $3$-colorable ($P_5$, gem)-free graphs with an induced $C_5$ below. Let $\mathcal{G}_1$ be the family of graphs, consisting of $G_1$ and all graphs obtained from $G_1$ by expanding one vertex to a disjoint union of $K_2$, 
$\mathcal{G}_4$ be the family of graphs obtained from $G_4$ by expanding $x_5$ to a disjoint union of $K_2$, and  $\mathcal{G}_{10}=\{G_{10}\}$.

\begin{lem}\label{3colorable}
If $G$ be a $3$-colorable ($P_5$, gem)-free graph with an induced $C_5$, then $G$ has a false twin or $G\in \mathcal{G}_1\cup \mathcal{G}_4\cup \mathcal{G}_{10}$.
\end{lem}

\begin{proof}
By Lemma~\ref{Pgem}, we have that $G\in \mathcal{H}$ or $G$ is a $P_4$-free expansion of either $G_1$, $G_2$,$\ldots$, $G_9$ or $G_{10}$. Suppose $G\in \mathcal{H}$.  Let $x_i\in A_i$ for $i\in \{1,2,\ldots,5\}$.
Note that $G[\{x_1,x_2,\ldots,x_5\}]$ is an induced $C_5$. Since $G$ is $3$-colorable,  at least one of $A_2$ and $A_5$ is an independent set by Observation~\ref{K3C5}. Therefore $G$ contains a false twin $(x, y)$ with $x\in A_2\cup A_5$ and $y\in A_6$. 

Now we assume that $G$ is a $P_4$-free expansion of $G_1, G_2,\ldots, G_9$ or $G_{10}$. In this case, if there exists a false twin, then we are done. So we suppose that there is no false twin in $G$. For any  $i\in \{1,2,\ldots,10\}$, since $G$ is $3$-colorable and gem-free,  any vertex of $G_i$ can only be expanded to a $2$-colorable $P_4$-free graph. 
We have the following.
\begin{claim}\label{Nx}
Any connected 2-colorable $P_4$-free graph except $K_1$ and $K_2$ contains a false twin.
\end{claim}
\begin{proof} Let $H$ be a connected $2$-colorable $P_4$-free graph and $H\notin\{K_1, K_2\}$. Then we can find an induced path $P$ of length $2$ in $H$. Let $P=xzy$. Suppose that $x$ has a neighbor $w$ other than $z$, then $w \notin N(z)$ because  $H$ is  $2$-colorable. So $w$ is adjacent to $y$, otherwise $H[\{w,x,z,y\}]$ is an induced  $P_4$ in $H$, a contradiction. So $N(x) \subseteq N(y)$ and $(x, y)$ is a false twin. 
\end{proof}

\begin{claim}\label{disjointK2}
For any  $i\in \{1,2,\ldots,10\}$, a vertex of $G_i$ can only be expanded to $K_1$ or a disjoint union of $K_2$.
\end{claim}
\begin{proof}
 Suppose that a vertex of $G_i$ is expanded to $H$, where $i\in\{1,2,\ldots,10\}$. By the definition of  expansion, we know that if $x,y\in V(H)$ and $N_{H}(x) \subseteq N_{H}(y)$, then $N_{G_i}(x) \subseteq N_{G_i}(y)$. So  we have that each component of $H$ is $K_1$ or $K_2$ by Claim \ref{Nx}. If $H$ contains a $K_1$, then  $H\cong K_1$, otherwise the neighborhood of the vertex in $K_1$ is included in the neighborhood of any other vertex in $H$.
\end{proof}

Since $G$ contains no false twin, for any $i\in\{1,2,\ldots,10\}$, if there exists a false twin $(x, y)$ in $G_i$, then $x$ has to be expanded to a disjoint union of $K_2$. Thus, for some $i\in\{1,2,\ldots,10\}$,  if  $G_i$ contains two false twins $(x, y)$ and $(u, v)$ such that $x$ and $u$
are on an induced $C_5$ (we use $[G_i:(x, y),(u, v)]$ to denote the special false twin pair), then by Observation~\ref{K3C5} and Claim~\ref{disjointK2}, $G$ is not a $P_4$-free expansion of $G_i$. Therefore $G$ is not a $P_4$-free expansion of $G_2, G_3, G_5, G_6, G_8, G_9$, as  these special false twin pairs are $[G_2:(x_5,x_6),(x_2,x_6)], [G_3:(x_5,x_6),(x_7,x_3)], [G_5:(x_5,x_7),(x_8,x_2)], [G_6:(x_5,x_6),(x_8,x_1)], [G_8:(x_2,x_7),(x_3,x_8)], [G_9:(x_2,x_7),(x_3,x_8)]$, respectively.
In $G_7$,  since $(x_8,x_5)$ and $(x_7,x_2)$ are two false twins, $x_8$ and $x_7$ must be expanded to a disjoint union of $K_2$. Note that $G_7[\{x_1,x_2,x_3\}]$ and $G_7[\{x_4,x_5,x_6\}]$ are two triangles. Since $G$ contains no $K_4$, any vertex in $\{x_1,x_2, \ldots, x_6\}$ cannot be expanded to a disjoint union of $K_2$ by Claim~\ref{disjointK2}. If  $G$ is the graph  obtained from $G_7$ by expanding each of $x_8$ and $x_7$ to a disjoint union of $K_2$, then  for any $3$-coloring $c$ of $G$, we have $c(x_2)= c(x_4)$ and $c(x_1)= c(x_5)$. Thus, we have $c(x_3)= c(x_6)$, a contradiction. Hence, $G$ is not a $P_4$-free expansion of $G_7$. 

Now we consider $G_1$, $G_4$ and $G_{10}$. Since $G_1$ is an induced $C_5$,  Observation~\ref{K3C5} and Claim \ref{disjointK2} imply that at most one vertex of $G_1$ can be expanded to a disjoint union of $K_2$. So $G\in {\mathcal{G}}_1$. In $G_4$, since  $(x_5,x_7)$ is a false twin, $x_5$ must be expanded to a disjoint union of $K_2$. Note that any  vertex of $G_4$ other than $x_5$  is in a triangle. Since $G$ contains no $K_4$, we have $G\in \mathcal{G}_4$ by Claim~\ref{disjointK2}. Finally for $G_{10}$, since any  vertex  of $G_{10}$ is in a triangle, $G\in \mathcal{G}_{10}$. 

This completes the proof of Lemma \ref{3colorable}.
\end{proof}



\subsection{Proof of Theorem \ref{mainthm1.1}}

\noindent By way of contradiction, suppose that there exists $k$ with $k\geq4$ such that $G$ is a $3$-colorable $P_5$-free graph that is not $k$-mixng. Among all $3$-colorable $P_5$-free graphs that is not $k$-mixng, we choose $G$ so that
$|V(G)|$ is minimum.

\begin{claim}\label{dG}
	 $\delta(G)\geq 3$.
\end{claim}
\begin{proof}
Suppose that $G$ contains a vertex $x$ with $d(x)\leq 2$. Let  $\alpha$ and $\beta$ be two
$k$-colorings of $G$.    Let  $\alpha'$ and $\beta'$ be the restrictions of $\alpha$ and $\beta$ to $G'=G-\{x\}$. By
the minimality of $G$, there exists a sequence $\mathcal{S}'$ of recolorings that transforms $\alpha'$ into $\beta'$. We extend $\mathcal{S}'$ to a sequence $\mathcal{S}$ of
recolorings in $G$. To form $\mathcal{S}$ in $G$, we can perform each recoloring step from $\mathcal{S}'$, except
when a neighbor $y$ of $x$ is to be recolored with the current color  of $x$. In that case, we
need to recolor $x$ before recoloring its neighbor $y$. The number of colors unused on $N[x]$ is at least $k-(d(x)+1)\geq 1$. We recolor $x$ with one of these colors that is not
the target color in the next recoloring of a neighbor of $x$.
Finally, if need, we
recolor $x$ to $\beta(x)$. Thus, we get that $G$ is $k$-mixing, a contradiction.
\end{proof}

\begin{claim}\label{nft} 
	 $G$ has no false twin.
\end{claim}
\begin{proof}
Suppose that $G$ contains a flase twin $(x,y)$.  Let  $\alpha$ and $\beta$ be any two
$k$-colorings of $G$ where $k\geq 4$.  Let  $\alpha'$ and $\beta'$ be the restrictions of $\alpha$ and $\beta$ to $G'=G-\{x\}$. By
the minimality of $G$, there exists a sequence $\mathcal{S}'$ of recolorings that transforms $\alpha'$ into $\beta'$. To extend $\mathcal{S}'$ to a sequence $\mathcal{S}$ of
recolorings in $G$, we first recolor $x$ to $\alpha(y)$, this is possible as $x\not\sim y$.  Then we  perform each recoloring step from $\mathcal{S}'$   except when $y$ is recolored. Every time, after recoloring $y$, we
need to recolor $x$ to  the current color of $y$. Finally, if need, we
recolor $x$ to $\beta(x)$. Since $N(x)\subseteq N(y)$, $\mathcal{S}$ is a sequence of recolorings that transforms $\alpha$ into $\beta$.
Thus, we get that $G$ is $k$-mixing, a contradiction.
\end{proof}

\begin{claim}\label{ngem}
	$G$ has no induced gem.
\end{claim}

\begin{proof}
Suppose that $G$ has an induced gem, see Figure~\ref{fig:gemproperty}(a). Let $\alpha$ be a $3$-coloring of $G$ with $\alpha(u)=3,\alpha(u_1)=\alpha(u_3)=1,\alpha(u_2)=\alpha(u_4)=2$. 
By Claim~\ref{nft}, we have  $N(u_1)\setminus N(u_3)\neq \emptyset$ and $N(u_4)\setminus N(u_2)\neq \emptyset$. 
Here we have the following property (P1).
\\[8pt]
\text{{\bf (P1)} For any $x\in N(u_1)\setminus N(u_3)$, we have $\alpha(x)=3$ and $N(x)\cap \{u_1,u_2,u_3,u_4\}=\{u_1,u_2,u_4\}$.}

\begin{proof} If $\alpha(x)\neq3$ or $N(x)\cap \{u_1,u_2,u_3,u_4\}=\{u_1\}$, then $xu_1u_2u_3u_4$ is an induced $P_5$, contradicting the fact that $G$ is $P_5$-free. 
Suppose $N(x)\cap \{u_1,u_2,u_3,u_4\}=\{u_1,u_2\}$.
Let $y\in N(u_4)\setminus N(u_2)$. Note that $x\neq y$. Then $xu_2uu_4y$ is an induced $P_5$,  a contradiction. Suppose $N_G(x)\cap \{u_1,u_2,u_3,u_4\}=\{u_1,u_4\}$. 
By Claim~\ref{nft}, there exists a vertex $z$ such that $z\in N(x)\setminus N(u)$. Then $zxu_1uu_3$ is an induced $P_5$ when $\alpha(z)=1$ and $zxu_4uu_2$ is an induced $P_5$ when $\alpha(z)=2$,  a contradiction. Therefore, $\alpha(x)=3$ and $N(x)\cap \{u_1,u_2,u_3,u_4\}=\{u_1,u_2,u_4\}$.
\end{proof}

Let $x\in N(u_1)\setminus N(u_3)$ and $y\in N(u_4)\setminus N(u_2)$. By (P1) and the symmetry of $u_1$ and $u_4$, we have $x\neq y$, $\alpha(y)=3$ and $N(y)\cap \{u_1,u_2,u_3,u_4\}=\{u_1,u_3,u_4\}$. For $i\in \{1,2\}$, let $N^i(x)$ (resp. $N^i(y)$) be the set of neighbors of $x$ (resp. $y$) colored $i$ under $\alpha$ except $u_1,u_2, u_3,u_4$. Note that  for $i\in \{1,2\}$, $N^i(x)$ or $N^i(y)$ might be empty. We have the following property (P2).
\\[8pt]
{\bf (P2)} For any $x_1\in N^1(x)$, we have $x_1\sim u$ and $x_1\sim y$. For any $y_2\in N^2(y)$, we have $y_2\sim u$ and $y_2\sim x$.

\begin{proof}
By the symmetry of $x_1$ and $y_2$, it suffices to prove $x_1\sim u$ and $x_1\sim y$. If $x_1\not\sim u$, then $x_1xu_1uu_3$ is an induced $P_5$,  a contradiction. If $x_1\not\sim y$, then $xx_1uu_3y$ is an induced $P_5$, a contradiction.
Therefore, $x_1\sim u$ and $x_1\sim y$.
\end{proof}

Since $x\not\sim u$ and $y\not\sim u$,  Claim~\ref{nft} and (P2) imply that $N^2(x)\setminus N(u)\neq \emptyset$ and $N^1(y)\setminus N(u)\neq \emptyset$.
 We have the following property (P3).
\\[8pt]
{\bf (P3)}For any $x_2^*\in N^2(x)\setminus N(u)$ and $y_1^*\in N^1(y)\setminus N(u)$, we have $x_2^*\not\sim y$, ${y_1}^*\not\sim x$, 
$x_2^*\sim y_1^*$, $x_2^*\sim u_1$,
$x_2^*\sim u_3$, $y_1^*\sim u_2$, $y_1^*\sim u_4$.

\begin{proof} 
Since $x_2^*\not\in N(u)$ and $y_1^*\not\in N(u)$, we have  $x_2^*\not\sim y$ and  $y_1^*\not\sim x$, otherwise, $x_2^*yu_4uu_2$ and $y_1^*xu_1uu_3$ are two induced $P_5$, a contradiction.
If $x_2^*\not\sim y_1^*$, then $x_2^*xu_2{y_1^*}y$ is an induced $P_5$, a contradiction. If $x_2^*\not\sim u_3$, then $x_2^*xu_2u_3y$ is an induced $P_5$, a contradiction. If $x_2^*\not\sim u_1$, then $y_1^*{x_2}^*u_3uu_1$ is an induced $P_5$, a contradiction. Thus $x_2^*\sim y_1^*$, $x_2^*\sim u_1$
and $x_2^*\sim u_3$.
Similarly, we have $y_1^*\sim u_2$ and  $y_1^*\sim u_4$.
\end{proof}

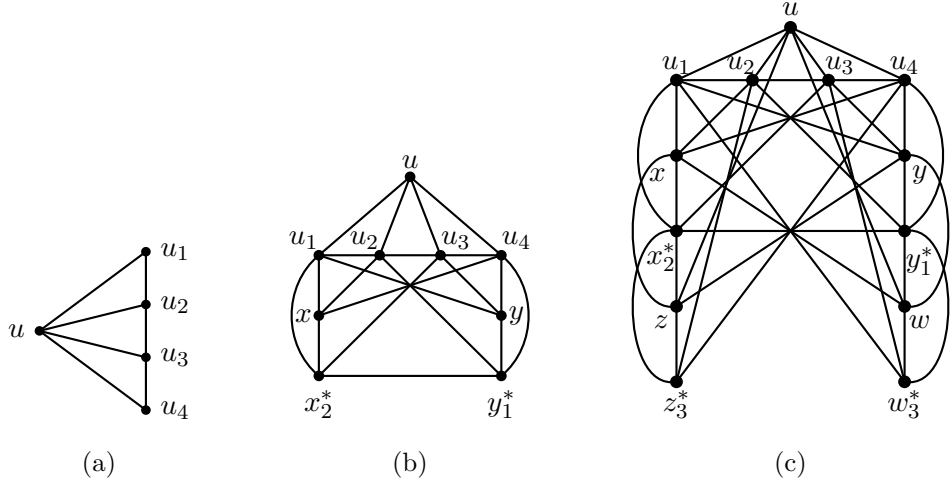
\begin{figure}[t]
	\centering
 \begin{subfigure}[t]{0.2\textwidth}
	\centering
	\begin{tikzpicture}[scale=0.7]   
		\draw [line width=0.8pt, black] (0,1.5) to (2,3);
		\draw [line width=0.8pt, black] (0,1.5) to (2,2);
		\draw [line width=0.8pt, black] (0,1.5) to (2,0);
		\draw [line width=0.8pt, black] (0,1.5) to (2,1);
		\draw [line width=0.8pt, black] (2,3) to (2,0);
	
		\draw [fill=black,line width=0.8pt] (0,1.5) node[left=0.5mm] {$u$} circle (2pt); 
		\draw [fill=black,line width=0.8pt] (2,3) node[right=0.5mm] {$u_1$} circle (2pt); 
		\draw [fill=black,line width=0.8pt] (2,2) node[right=0.5mm] {$u_2$} circle (2pt); 
		\draw [fill=black,line width=0.8pt] (2,1) node[right=0.5mm] {$u_3$} circle (2pt); 
		\draw [fill=black,line width=0.8pt] (2,0) node[right=0.5mm] {$u_4$} circle (2pt); 
	\end{tikzpicture}
	\caption{}
\end{subfigure}
   \begin{subfigure}[t]{0.32\textwidth}
   	\centering
   	\begin{tikzpicture}[scale=0.8]   
		\draw [line width=0.8pt, black] (1.5,3.3) to (1,2);
		\draw [line width=0.8pt, black] (1.5,3.3) to (2,2);
		\draw [line width=0.8pt, black] (1.5,3.3) to (3,2);
		\draw [line width=0.8pt, black] (1.5,3.3) to (0,2);
		\draw [line width=0.8pt, black] (0,2) to (3,2);
		\draw [line width=0.8pt, black] (0,1) to (0,2);
		\draw [line width=0.8pt, black] (0,1) to (1,2);
		\draw [line width=0.8pt, black] (0,1) to (3,2);
		\draw [line width=0.8pt, black] (3,1) to (0,2);
		\draw [line width=0.8pt, black] (3,1) to (2,2);
		\draw [line width=0.8pt, black] (3,1) to (3,2);
		\draw [line width=0.8pt, black] (3,1) to (3,0);
		\draw [line width=0.8pt, black] (0,1) to (0,0);
		\draw [line width=0.8pt, black] (0,0) to (2,2);
		\draw [line width=0.8pt, black] (0,0) to (3,0);
		\draw [line width=0.8pt, black] (3,0) to (1,2);
		\draw [bend right=50, line width=0.8pt, black] (3,0) to (3,2);
		\draw [bend left=50, line width=0.8pt, black] (0,0) to (0,2);
		
		\draw [fill=black,line width=0.8pt] (1.5,3.3)  circle (2pt); 
  \node at (1.5,3.55) {$u$};
		\draw [fill=black,line width=0.8pt] (0,2)  circle (2pt);
            \node at (-0.25,2.25) {$u_1$};
		\draw [fill=black,line width=0.8pt] (1,2)  circle (2pt); 
            \node at (0.75,2.25) {$u_2$};
		\draw [fill=black,line width=0.8pt] (2,2)  circle (2pt); 
             \node at (2.25,2.25) {$u_3$};
		\draw [fill=black,line width=0.8pt] (3,2)  circle (2pt);
            \node at (3.25,2.25) {$u_4$};
		\draw [fill=black,line width=0.8pt] (0,1)  circle (2pt);
        \node at (-0.25,1) {$x$};
		\draw [fill=black,line width=0.8pt] (3,1)  circle (2pt);
        \node at (3.25,1) {$y$};
		\draw [fill=black,line width=0.8pt] (0,0) node[below=0.5mm] {${x^*_2}$} circle (2pt);
		\draw [fill=black,line width=0.8pt] (3,0) node[below=0.5mm] {${y^*_1}$} circle (2pt); 
	\end{tikzpicture}
	\caption{}
    \end{subfigure}
    \begin{subfigure}[t]{0.32\textwidth}
	   \centering
	\begin{tikzpicture}[scale=1.0]   
		\draw [line width=0.8pt, black] (1.5,2.7) to (1,2);
		\draw [line width=0.8pt, black] (1.5,2.7) to (2,2);
		\draw [line width=0.8pt, black] (1.5,2.7) to (3,2);
		\draw [line width=0.8pt, black] (1.5,2.7) to (0,2);
		\draw [line width=0.8pt, black] (0,2) to (3,2);
		\draw [line width=0.8pt, black] (0,1) to (0,2);
		\draw [line width=0.8pt, black] (0,1) to (1,2);
		\draw [line width=0.8pt, black] (0,1) to (3,2);
		\draw [line width=0.8pt, black] (3,1) to (0,2);
		\draw [line width=0.8pt, black] (3,1) to (2,2);
		\draw [line width=0.8pt, black] (3,1) to (3,2);
		\draw [line width=0.8pt, black] (3,1) to (3,0);
		\draw [line width=0.8pt, black] (0,1) to (0,0);
		\draw [line width=0.8pt, black] (0,0) to (2,2);
		\draw [bend left=60, line width=0.8pt, black] (0,0) to (0,2);
		\draw [line width=0.8pt, black] (0,0) to (3,0);
		\draw [line width=0.8pt, black] (3,0) to (1,2);
		\draw [bend right=60, line width=0.8pt, black] (3,0) to (3,2);
       \draw [bend left=100, line width=0.8pt, black] (0,-1) to (0,1);
       \draw [line width=0.8pt, black] (0,-1) to (1.5,2.7);
       \draw [line width=0.8pt, black] (0,-1) to (3,1);
		\draw [line width=0.8pt, black] (0,0) to (0,-1);
		\draw [line width=0.8pt, black] (0,-1) to (0,-2);
		\draw [line width=0.8pt, black] (3,0) to (3,-1);
		\draw [line width=0.8pt, black] (3,-1) to (3,-2);
		\draw [line width=0.8pt, black] (0,-2) to (3,2);
        \draw [bend left=100, line width=0.8pt, black] (0,-2) to (0,0);
        \draw [line width=0.8pt, black] (0,-2) to (1,2);
		\draw [line width=0.8pt, black] (3,-2) to (0,2);
        \draw [bend right=100, line width=0.8pt, black] (3,-1) to (3,1);
        \draw [line width=0.8pt, black] (3,-1) to (1.5,2.7);
        \draw [line width=0.8pt, black] (3,-1) to (0,1);
        \draw [line width=0.8pt, black] (3,-2) to (2,2);
        \draw [bend right=100, line width=0.8pt, black] (3,-2) to (3,0);

		\draw [fill=black,line width=0.8pt] (1.5,2.7)
  circle (2pt); 
  \node at (1.5,2.95) {$u$}; 
		\draw [fill=black,line width=0.8pt] (0,2)  circle (2pt); 
  \node at (0,2.2) {$u_1$};
		\draw [fill=black,line width=0.8pt] (1,2)  circle (2pt); 
            \node at (0.85,2.2) {$u_2$};
		\draw [fill=black,line width=0.8pt] (2,2) circle (2pt); 
            \node at (2.15,2.2) {$u_3$};
		\draw [fill=black,line width=0.8pt] (3,2) circle (2pt);
  \node at (3,2.2) {$u_4$};
		\draw [fill=black,line width=0.8pt] (0,1)  circle (2pt);
        \node at (-0.25,0.75) {$x$};
		\draw [fill=black,line width=0.8pt] (3,1) circle (2pt);
        \node at (3.2,0.75) {$y$};
		\draw [fill=black,line width=0.8pt] (0,0)  circle (2pt);
  \node at (-0.2,-0.35) {${x^*_2}$};
		\draw [fill=black,line width=0.8pt] (3,0) circle (2pt); 
   \node at (3.2,-0.4) {${y^*_1}$};
		\draw [fill=black,line width=0.8pt] (0,-1)  circle (2pt);
  \node at (-0.2,-1.2) {$z$};
		\draw [fill=black,line width=0.8pt] (3,-1)  circle (2pt); 
   \node at (3.2,-1.2) {$w$};
		\draw [fill=black,line width=0.8pt] (0,-2)  circle (2pt); 
  \node at (0,-2.3) {$z_3^*$};
		\draw [fill=black,line width=0.8pt] (3,-2)  circle (2pt); 
        \node at (3,-2.3) {$w_3^*$};
	\end{tikzpicture}
	\caption{}
    \end{subfigure}
\caption{Structural properties of gems.}
	\label{fig:gemproperty}
\end{figure}

Let $x_2^*\in N^2(x)\setminus N(u)$ and $y_1^*\in N^1(y)\setminus N(u)$. Since $\alpha(x_2^*)=\alpha(u_2)$ and $\alpha(y_1^*)=\alpha(u_3)$, we have $x_2^*\not\sim u_2$ and $y_1^*\not\sim u_3$. The induced subgraph  $H=G[\{u,u_1,u_2,u_3,u_4,x,y,x_2^*,y_1^*\}]$ is shown in Figure~\ref{fig:gemproperty}(b). Note that $G[\{u_1,u,u_2,x,x_2^*\}]$ and $G[\{u_4,u,u_3,y,y_1^*\}]$ are two induced gems. We first consider $G[\{u_1,u,u_2,x,x_2^*\}]$. The following discussion is similar to the foregoing discussion for the structure of gem.  
Since $x_2^*\not\sim u_2$, by Claim~\ref{nft} there exists a vertex $z$ such that $z\in N(x_2^*)\setminus N({u_2})$. By (P1), we know that $z\sim x$, $z\sim u$ and $\alpha(z)=1$. By the adjacency relationship, we have $z\neq u_3, y_1^*$, which means that $z\not\in V(H)$.  Since $z\in N^1(x)$, we have $z\sim y$ by (P2). Since $\alpha(z)=\alpha(u_1)=1$, we have $z\not\sim u_1$. Then  Claim~\ref{nft} and (P2) imply  that there exists a vertex  $z_3^*\in N(z)\setminus N(u_1)$ and $\alpha(z_3^*)=3$.
Since $y\sim u_1$, we have $z_3^*\neq y$. Thus, $z_3^*\not\in V(H)$. Note that $u_3\in N(u)\setminus N(u_1)$,  $u_4\in N(u_3)\setminus N(u_1)$ and $u_4\not\sim u_1$. Then by (P3), we have $z^*_3\sim x^*_2$, $z^*_3\sim u_2$ and $z^*_3\sim u_4$. Now we consider the induced gem $G[\{u_4,u,u_3,y,y_1^*\}]$. By symmetry, 
 there exist two vertices $w, w_3^*\not\in V(H)$ such that $w\in N(y_1^*)\setminus N(u_4)$, $\alpha(w)=2$ and $w_3^*\in N(w)\setminus N(u_4)$, $\alpha(w_3^*)=3$, 
 $w_3^*\sim u_1$. 
  Then  $w_3^*u_1u u_4{z_3}^*$ is an induced $P_5$,   which is shown in Figure~\ref{fig:gemproperty}(c), a contradiction.
\end{proof}

\begin{claim}\label{nC5} 
	$G$ has no induced $C_5$.
\end{claim}
\begin{proof}
Suppose that $G$ has an induced $C_5$. By Claim~\ref{ngem}, we know that $G$ is a $3$-colorable ($P_5$, gem)-free graph with an induced $C_5$.
Then by Lemma~\ref{3colorable}, we have $G\in \mathcal{G}_1\cup \mathcal{G}_4\cup \mathcal{G}_{10}$. Note that any graph in $\mathcal{G}_1\cup \mathcal{G}_4\cup \mathcal{G}_{10}$ has no false twin. Since any graph in ${\mathcal{G}}_1$ has minimum degree $2$,  Claim~\ref{dG} implies that $G\notin {\mathcal{G}}_1$. Now we consider $G\in  \mathcal{G}_4\cup \mathcal{G}_{10}$. We first give a subclaim as follows.

\begin{subclaim}\label{A}
Let $x,y$ be two nonadjacent vertices of $G$ and $G'$ be the graph obtained from $G$ by identifying $x$ and $y$  and deleting parallel edges. If $G'$ is a $k$-colorable $P_5$-free graph, then there exists a $k$-coloring $\alpha$ of $G$ such that $\alpha$ cannot be transformed to any $k$-coloring $\alpha'$ of $G$ with $\alpha'(x)=\alpha'(y)$. 
\end{subclaim}

\noindent{\it Proof.}
Let $\alpha$ and $\beta$ be any two $k$-colorings of $G$. Suppose that $\alpha$ and $\beta$ can be transformed to two $k$-colorings $\alpha'$ and $\beta'$ of $G$ such that $\alpha'(x)=\alpha'(y)$ and $\beta'(x)=\beta'(y)$, respectively.
Let $z$ denote the  new vertex in $G'$ after identifying $x$ and $y$ in $G$. Let $\alpha''$ and $\beta''$ be two $k$-colorings of $G'$ satisfying that $\alpha''(z)=\alpha'(x)$, $\beta''(z)=\beta'(x)$ and $\alpha''(w)=\alpha'(w)$, $\beta''(w)=\beta'(w)$ for any $w\in V(G')\setminus \{z\}$. 
Since $G'$ is $P_5$-free and $k$-colorable,   $\alpha''$ can be transformed into $\beta''$  by
the minimality of $G$. This implies that $\alpha'$ can be transformed into $\beta'$. Thus, $\alpha$ can be transformed into $\beta$ by $\alpha$ into $\alpha'$, $\alpha'$ into $\beta'$, $\beta'$ into $\beta$. Hence, by arbitrary of $\alpha$ and $\beta$, $G$ is $k$-mixing, a contradiction.\qed

The following discussion is split into two cases below.

{\bf Case 1.} $G\in  {\mathcal{G}}_4$. See Figure~\ref{g4}(a).

Let $G'$ be the graph obtained from $G$ by identifying $x_1$ and $x_6$ and deleting parallel edges as shown in Figure~\ref{g4}(b).  It is worth noting that $G'$ is $P_5$-free. A $3$-coloring of $G'$ is shown in Figure~\ref{g4}(c). So $G'$ is $P_5$-free and $3$-colorable.
Now we claim that any $k$-coloring $\alpha$ of $G$ can be transformed to a $k$-coloring $\alpha'$ of $G$ such that $\alpha'(x_1)=\alpha'(x_6)$.	 If $\alpha(x_1)=\alpha(x_6)$, then we are done. 
So we assume  $\alpha(x_1)\neq\alpha(x_6)$. Without loss of generality, we assume that $\alpha(x_1)=1$ and $\alpha(x_6)=2$. Let $U=V(G)\setminus\{x_1,x_2,x_3,x_6\}$. Then 
for any $x\in U$, since $x$ is the common neighbor of $x_1$ and $x_6$, we have $\alpha(x)\in\{3,4,\ldots,k\}$. If $\alpha(x_2)\neq 2$, then we can recolor $x_1$ by color $2$ and we are done, so we assume $\alpha(x_2)=2$. By symmetry, we assume $\alpha(x_3)=1$. Note that $k\geq 4$. Let $a\in \{3,4,\ldots,k\}\setminus \alpha(x_4)$. We recolor $x_3$, $x_6$ by colors $a$, $1$ in order, which yields a $k$-coloring $\alpha'$ of $G$ such that $\alpha'(x_1)=\alpha'(x_6)$, contradicting Subclaim~\ref{A}.

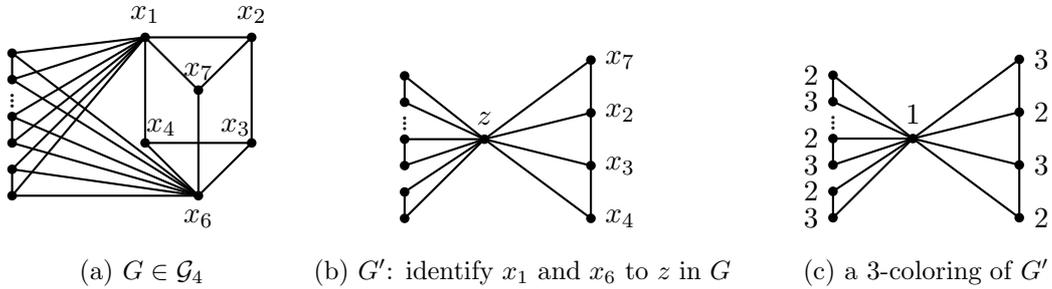
\begin{figure}[htbp]
	\centering
   \begin{subfigure}[t]{0.28\textwidth}
   	\centering
   	\begin{tikzpicture}[scale=0.7]   
		\draw [line width=0.8pt, black] (1,3) to (3,3);
		\draw [line width=0.8pt, black] (1,3) to (2,2);
		\draw [line width=0.8pt, black] (3,3) to (2,2);
		\draw [line width=0.8pt, black] (1,3) to (1,1);
		\draw [line width=0.8pt, black] (3,3) to (3,1);
		\draw [line width=0.8pt, black] (1,1) to (3,1);
		\draw [line width=0.8pt, black] (1,1) to (2,0);
		\draw [line width=0.8pt, black] (3,1) to (2,0);
		\draw [line width=0.8pt, black] (-1.5,0) to (2,0);
		\draw [line width=0.8pt, black] (-1.5,0) to (1,3);
		\draw [line width=0.8pt, black] (2,0) to (2,2);
        \draw [line width=0.8pt, black] (-1.5,0.5) to (-1.5,0);
        \draw [line width=0.8pt, black] (2,0) to (-1.5,0.5);
        \draw [line width=0.8pt, black] (-1.5,0.5) to (1,3);
        \draw [line width=0.8pt, black] (-1.5,1.5) to (1,3);
        \draw [line width=0.8pt, black] (-1.5,1) to (1,3);
        \draw [line width=0.8pt, black] (-1.5,1) to (2,0);
        \draw [line width=0.8pt, black] (-1.5,1.5) to (2,0);
        \draw [line width=0.8pt, black] (-1.5,1.5) to (-1.5,1);
        \draw [line width=0.8pt, black] (-1.5,2.2) to (-1.5,2.7);
        \draw [line width=0.8pt, black] (-1.5,2.2) to (1,3);
        \draw [line width=0.8pt, black] (-1.5,2.2) to (2,0);
        \draw [line width=0.8pt, black] (-1.5,2.7) to (2,0);
        \draw [line width=0.8pt, black] (-1.5,2.7) to (1,3);
        
        \draw [fill=black,line width=0.8pt] (-1.5,2.2)  circle (2pt);
        \draw [fill=black,line width=0.8pt] (-1.5,2.7)  circle (2pt);
        \draw [fill=black,line width=0.8pt] (-1.5,1.7)  circle (0.25pt);
        \draw [fill=black,line width=0.8pt] (-1.5,1.8)  circle (0.25pt);
        \draw [fill=black,line width=0.8pt] (-1.5,1.9)  circle (0.25pt);
        \draw [fill=black,line width=0.8pt] (-1.5,1.5)  circle (2pt); 
		\draw [fill=black,line width=0.8pt] (-1.5,0)  circle (2pt); 
		\draw [fill=black,line width=0.8pt] (-1.5,1)  circle (2pt); 
        \draw [fill=black,line width=0.8pt] (-1.5,0.5)  circle (2pt);          
		\draw [fill=black,line width=0.8pt] (2,0) node[below=0.5mm] {$x_6$} circle (2pt); 
		\draw [fill=black,line width=0.8pt] (1,1)  circle (2pt); 
        \node at (1.3,1.3) {$x_4$};
		\draw [fill=black,line width=0.8pt] (3,1)  circle (2pt); 
        \node at (2.7,1.3) {$x_3$};
		\draw [fill=black,line width=0.8pt] (1,3) node[above=0.5mm] {$x_1$} circle (2pt);
		\draw [fill=black,line width=0.8pt] (3,3) node[above=0.5mm] {$x_2$} circle (2pt);
		\draw [fill=black,line width=0.8pt] (2,2) circle (2pt);
        \node at (2,2.35) {$x_7$};
		
	\end{tikzpicture}
	\caption{$G\in{\mathcal{G}}_4$}
	\label{fig:g4}
\end{subfigure}
\begin{subfigure}[t]{0.36\textwidth}
   	\centering
   	\begin{tikzpicture}[scale=0.7]   
		\draw [line width=0.8pt, black] (1,1.5) to (3,3);
		\draw [line width=0.8pt, black] (3,3) to (3,0);
		\draw [line width=0.8pt, black] (-0.5,0) to (1,1.5);
        \draw [line width=0.8pt, black] (-0.5,0.5) to (-0.5,0);
        \draw [line width=0.8pt, black] (-0.5,0.5) to (1,1.5);
        \draw [line width=0.8pt, black] (-0.5,1.5) to (1,1.5);
        \draw [line width=0.8pt, black] (-0.5,1) to (1,1.5);
        \draw [line width=0.8pt, black] (-0.5,1.5) to (-0.5,1);
        \draw [line width=0.8pt, black] (-0.5,2.2) to (-0.5,2.7);
        \draw [line width=0.8pt, black] (-0.5,2.2) to (1,1.5);
        \draw [line width=0.8pt, black] (-0.5,2.7) to (1,1.5);
        \draw [line width=0.8pt, black] (1,1.5) to (3,0);
        \draw [line width=0.8pt, black] (1,1.5) to (3,1);
         \draw [line width=0.8pt, black] (1,1.5) to (3,2);
         
        \draw [fill=black,line width=0.8pt] (-0.5,2.2)  circle (2pt);
        \draw [fill=black,line width=0.8pt] (-0.5,2.7)  circle (2pt);
        \draw [fill=black,line width=0.8pt] (-0.5,1.7)  circle (0.25pt);
        \draw [fill=black,line width=0.8pt] (-0.5,1.8)  circle (0.25pt);
        \draw [fill=black,line width=0.8pt] (-0.5,1.9)  circle (0.25pt);
        \draw [fill=black,line width=0.8pt] (-0.5,1.5)  circle (2pt); 
		\draw [fill=black,line width=0.8pt] (-0.5,0)  circle (2pt); 
		\draw [fill=black,line width=0.8pt] (-0.5,1)  circle (2pt); 
        \draw [fill=black,line width=0.8pt] (-0.5,0.5)  circle (2pt);          
		\draw [fill=black,line width=0.8pt] (3,1) node[right=0.5mm] {$x_3$} circle (2pt); 
		\draw [fill=black,line width=0.8pt] (3,0) node[right=0.5mm] {$x_4$} circle (2pt); 
		\draw [fill=black,line width=0.8pt] (1,1.5) node[above=0.5mm] {$z$} circle (2pt);
		\draw [fill=black,line width=0.8pt] (3,2) node[right=0.5mm] {$x_2$} circle (2pt);
		\draw [fill=black,line width=0.8pt] (3,3) node[right=0.5mm] {$x_7$} circle (2pt);

	\end{tikzpicture}
	\caption{$G'$: identify $x_1$ and $x_6$ to $z$ in $G$}
	\label{g4'}
\end{subfigure}
\begin{subfigure}[t]{0.32\textwidth}
   	\centering
   	\begin{tikzpicture}[scale=0.7]   
		\draw [line width=0.8pt, black] (1,1.5) to (3,3);
		\draw [line width=0.8pt, black] (3,3) to (3,0);
		\draw [line width=0.8pt, black] (-0.5,0) to (1,1.5);
        \draw [line width=0.8pt, black] (-0.5,0.5) to (-0.5,0);
        \draw [line width=0.8pt, black] (-0.5,0.5) to (1,1.5);
        \draw [line width=0.8pt, black] (-0.5,1.5) to (1,1.5);
        \draw [line width=0.8pt, black] (-0.5,1) to (1,1.5);
        \draw [line width=0.8pt, black] (-0.5,1.5) to (-0.5,1);
        \draw [line width=0.8pt, black] (-0.5,2.2) to (-0.5,2.7);
        \draw [line width=0.8pt, black] (-0.5,2.2) to (1,1.5);
        \draw [line width=0.8pt, black] (-0.5,2.7) to (1,1.5);
        \draw [line width=0.8pt, black] (1,1.5) to (3,0);
        \draw [line width=0.8pt, black] (1,1.5) to (3,1);
         \draw [line width=0.8pt, black] (1,1.5) to (3,2);
         
        \draw [fill=black,line width=0.8pt] (-0.5,2.2) node[left=0.5mm] {$3$} circle (2pt);
        \draw [fill=black,line width=0.8pt] (-0.5,2.7) node[left=0.5mm] {$2$}  circle (2pt);
        \draw [fill=black,line width=0.8pt] (-0.5,1.7)  circle (0.25pt);
        \draw [fill=black,line width=0.8pt] (-0.5,1.8)  circle (0.25pt);
        \draw [fill=black,line width=0.8pt] (-0.5,1.9)  circle (0.25pt);
        \draw [fill=black,line width=0.8pt] (-0.5,1.5) node[left=0.5mm] {$2$}  circle (2pt); 
		\draw [fill=black,line width=0.8pt] (-0.5,0) node[left=0.5mm] {$3$}  circle (2pt); 
		\draw [fill=black,line width=0.8pt] (-0.5,1) node[left=0.5mm] {$3$}  circle (2pt); 
        \draw [fill=black,line width=0.8pt] (-0.5,0.5) node[left=0.5mm] {$2$}  circle (2pt);          
		\draw [fill=black,line width=0.8pt] (3,1) node[right=0.5mm] {$3$} circle (2pt); 
		\draw [fill=black,line width=0.8pt] (3,0) node[right=0.5mm] {$2$} circle (2pt); 
		\draw [fill=black,line width=0.8pt] (1,1.5) node[above=0.5mm] {$1$} circle (2pt);
		\draw [fill=black,line width=0.8pt] (3,2) node[right=0.5mm] {$2$} circle (2pt);
		\draw [fill=black,line width=0.8pt] (3,3) node[right=0.5mm] {$3$} circle (2pt);

	\end{tikzpicture}
	\caption{a $3$-coloring of $G'$}
	\label{g4'}
\end{subfigure}
\caption{Some related graphs in the proof of Case 1.}
\label{g4}
\end{figure}

{\bf Case 2.} $G\in  {\mathcal{G}}_{10}$. See Figure~\ref{g4g10}(a).

First, we redraw $G_{10}$ as Figure~\ref{g4g10}(a). Let $G'$ be the graph obtained from $G$ by identifying $x_2$ and $x_6$  and deleting parallel edges. Let $z$ denote the  new vertex in $G'$. As shown in Figure~\ref{g4g10}(b). Figure~\ref{g4g10}(c) gives a $3$-coloring of $G'$, so $G'$ is $3$-colorable. If $G'$ contains an induced $P_5$, then it must contain the edge $x_4z$ because $G-\{x_6\}$ is $P_5$-free. In addition, $G'$  cannot contain both $x_8$ and $x_3$, because they are the common neighbors of  $z$ and $x_4$. However, the graph $G'-\{x_3,x_8\}$ is obviously $P_5$-free, a contradiction. Hence, $G'$ is $P_5$-free.

\begin{figure}[htbp]
	\centering
\begin{subfigure}[t]{.3\textwidth}
	\centering
	\begin{tikzpicture}[scale=1.2]   
		\draw [line width=0.8pt, black] (3,0) to (0,0);
		\draw [line width=0.8pt, black] (0,0) to (0,1);
		\draw [line width=0.8pt, black] (0,0) to (2,1);
		\draw [line width=0.8pt, black] (0,0) to (3,1);
		\draw [line width=0.8pt, black] (0,1) to (3,1);
		\draw [line width=0.8pt, black] (1,0) to (1,1);
		\draw [line width=0.8pt, black] (2,0) to (2,1);
		\draw [line width=0.8pt, black] (3,0) to (3,1);
		\draw [line width=0.8pt, black] (3,0) to (0,1);
		\draw [line width=0.8pt, black] (3,0) to (1,1);
		\draw [line width=0.8pt, black] (0,1) to (1.5,-1);
		\draw [line width=0.8pt, black] (1,0) to (1.5,-1);
		\draw [line width=0.8pt, black] (2,0) to (1.5,-1);
		\draw [line width=0.8pt, black] (3,1) to (1.5,-1);

		\draw [fill=black,line width=0.8pt] (0,0) node[below=0.5mm] {$x_4$} circle (1.5pt); 
		\draw [fill=black,line width=0.8pt] (0,1) node[above=0.5mm] {$x_5$} circle (1.5pt); 
		\draw [fill=black,line width=0.8pt] (1,0) circle (1.5pt); 
           \node at (1.25,0.15) {$x_3$};
		\draw [fill=black,line width=0.8pt] (1,1) node[above=0.5mm] {$x_7$} circle (1.5pt); 
		\draw [fill=black,line width=0.8pt] (2,0)  circle (1.5pt);
            \node at (1.75,0.15) {$x_2$};
		\draw [fill=black,line width=0.8pt] (2,1) node[above=0.5mm] {$x_8$} circle (1.5pt);
		\draw [fill=black,line width=0.8pt] (3,0) node[below=0.5mm] {$x_1$} circle (1.5pt);
		\draw [fill=black,line width=0.8pt] (3,1) node[above=0.5mm] {$x_6$} circle (1.5pt);
		\draw [fill=black,line width=0.8pt] (1.5,-1) node[below=0.5mm] {$x_9$} circle (1.5pt);

	\end{tikzpicture}
	\caption{$G=G_{10}$}
	\label{fig:g10}
\end{subfigure}
\begin{subfigure}[t]{0.37\textwidth}
	\centering
	\begin{tikzpicture}[scale=1.2]   
		\draw [line width=0.8pt, black] (3,0) to (0,0);
		\draw [line width=0.8pt, black] (0,0) to (0,1);
		\draw [line width=0.8pt, black] (0,0) to (2,1);
		\draw [line width=0.8pt, black] (1,0) to (1,1);
		\draw [line width=0.8pt, black] (2,0) to (2,1);
		\draw [line width=0.8pt, black] (3,0) to (0,1);
		\draw [line width=0.8pt, black] (3,0) to (1,1);
		\draw [line width=0.8pt, black] (0,1) to (1.5,-1);
		\draw [line width=0.8pt, black] (1,0) to (1.5,-1);
		\draw [line width=0.8pt, black] (2,0) to (1.5,-1);
        \draw [line width=0.8pt, black] (0,1) to (2,1);
        \draw [bend right=50, line width=0.8pt, black] (0,0) to (2,0);
		
		\draw [fill=black,line width=0.8pt] (0,0) node[below=0.5mm] {$x_4$} circle (1.5pt); 
		\draw [fill=black,line width=0.8pt] (0,1) node[above=0.5mm] {$x_5$} circle (1.5pt); 
		\draw [fill=black,line width=0.8pt] (1,0) circle (1.5pt); 
       \node at (1.2,0.2) {$x_3$};
		\draw [fill=black,line width=0.8pt] (1,1) node[above=0.5mm] {$x_7$} circle (1.5pt); 
		\draw [fill=black,line width=0.8pt] (2,0)  circle (1.5pt);
        \node at (1.8,0.18) {$z$};
		\draw [fill=black,line width=0.8pt] (2,1) node[above=0.5mm] {$x_8$} circle (1.5pt); 
		\draw [fill=black,line width=0.8pt] (3,0) node[below=0.5mm] {$x_1$} circle (1.5pt);
		\draw [fill=black,line width=0.8pt] (1.5,-1) node[below=0.5mm] {$x_9$} circle (1.5pt);
	\end{tikzpicture}
	\caption{$G'$: identify $x_2$ and $x_6$ to $z$ in $G$}
	\label{fig:g'}
\end{subfigure}
\begin{subfigure}[t]{0.3\textwidth}
	\centering
	\begin{tikzpicture}[scale=1.2]   
		\draw [line width=0.8pt, black] (3,0) to (0,0);
		\draw [line width=0.8pt, black] (0,0) to (0,1);
		\draw [line width=0.8pt, black] (0,0) to (2,1);
		\draw [line width=0.8pt, black] (1,0) to (1,1);
		\draw [line width=0.8pt, black] (2,0) to (2,1);
		\draw [line width=0.8pt, black] (3,0) to (0,1);
		\draw [line width=0.8pt, black] (3,0) to (1,1);
		\draw [line width=0.8pt, black] (0,1) to (1.5,-1);
		\draw [line width=0.8pt, black] (1,0) to (1.5,-1);
		\draw [line width=0.8pt, black] (2,0) to (1.5,-1);
        \draw [line width=0.8pt, black] (0,1) to (2,1);
        \draw [bend right=50, line width=0.8pt, black] (0,0) to (2,0);
		
		\draw [fill=black,line width=0.8pt] (0,0)  circle (1.5pt); 
        \node at (-0.15,-0.15) {$3$};
		\draw [fill=black,line width=0.8pt] (0,1) node[above=0.5mm] {$1$} circle (1.5pt); 
		\draw [fill=black,line width=0.8pt] (1,0) circle (1.5pt); 
       \node at (0.9,0.2) {$1$};
		\draw [fill=black,line width=0.8pt] (1,1) node[above=0.5mm] {$2$} circle (1.5pt); 
		\draw [fill=black,line width=0.8pt] (2,0) circle (1.5pt);
        \node at (1.9,0.18) {$2$};
		\draw [fill=black,line width=0.8pt] (2,1) node[above=0.5mm] {$1$} circle (1.5pt); 
		\draw [fill=black,line width=0.8pt] (3,0) node[below=0.5mm] {$3$} circle (1.5pt);
		\draw [fill=black,line width=0.8pt] (1.5,-1) node[below=0.5mm] {$3$} circle (1.5pt);
	\end{tikzpicture}
	\caption{a $3$-coloring of $G'$}
	\label{fig:g10color}
\end{subfigure}
\caption{Some related graphs in the proof of  Case 2.}
\label{g4g10}
\end{figure}
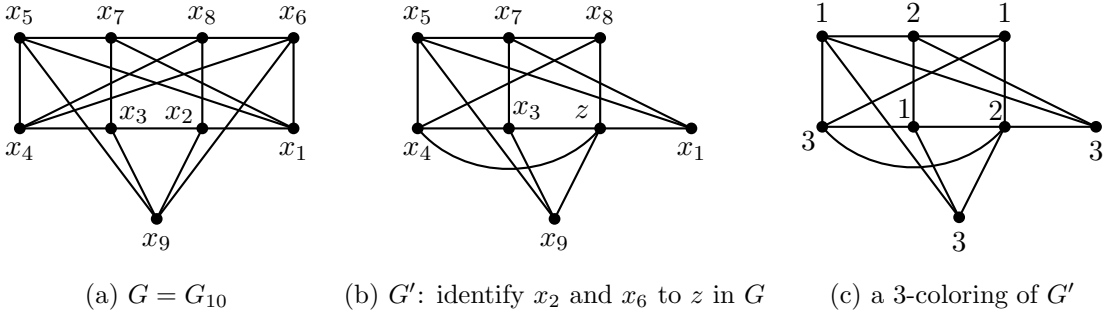

Next, we claim that any $k$-coloring $\alpha$ of $G$ can be transformed to a $k$-coloring $\alpha'$ of $G$ such that $\alpha'(x_2)=\alpha'(x_6)$.	 If $\alpha(x_2)=\alpha(x_6)$, then we are done. 
So we assume  $\alpha(x_2)\neq\alpha(x_6)$. Without loss of generality, we assume that $\alpha(x_2)=1$ and $\alpha(x_6)=2$. If  $\alpha(x_3)\neq 2$ or $\alpha(x_4)\neq 1$, then recolor $x_2$ by color $2$ or  $x_6$ by  color $1$, which yields a $k$-coloring $\alpha'$ of $G$ with $\alpha'(x_2)=\alpha'(x_6)$. So we suppose that $\alpha(x_3)= 2$ and $\alpha(x_4)= 1$.
If $\{\alpha(x_5),\alpha(x_8)\} \neq \{3,4\}$, then recolor $x_4$ by a color in $\{3,4\}\setminus \{\alpha(x_5), \alpha(x_8)\}$ and  $x_6$ by color $1$  in order, which yields a $k$-coloring $\alpha'$ of $G$ with $\alpha'(x_2)=\alpha'(x_6)=1$.  If $\{\alpha(x_5),\alpha(x_8)\} =\{3,4\}$, then $\alpha(x_7)=1$. Recolor $x_3$ by a color in $\{3,4\}\setminus \{\alpha(x_9)\}$ and  $x_2$ by color $2$ in order, which yields a $k$-coloring $\alpha'$ of $G$ with $\alpha'(x_2)=\alpha'(x_6)=2$, contradicting Subclaim~\ref{A}.
\end{proof}

\begin{claim}\label{nP5} 
$G$ has no induced $\overline{P_5}$.
\end{claim}
\begin{proof}
We first give a subclaim. Let  $[\{x,y\},x_1x_2x_3x_4]$ denote the graph structure  consisting of an induced $P_4=x_1x_2x_3x_4$ and two vertices $x$ and $y$ such that $\{x_1,x_2\}\subseteq N(x)\cap N(y)$ (see Figure~\ref{2cases}(a)). Note that the adjacency relationship between $x$ and $y$, and between $\{x,y\}$ and $\{x_3,x_4\}$ is uncertain. 
Let  $\mathcal{F}_G=\{[\{x,y\},x_1x_2x_3x_4]| G$
contains the  graph structure
$[\{x,y\},x_1x_2x_3x_4]\}$.

\begin{subclaim}\label{B}  $\mathcal{F}_G$ is an empty set.
\end{subclaim}

\begin{proof} 
Suppose that $G$ contains a graph structure $[\{x,y\},x_1x_2x_3x_4]\in\mathcal{F}_G$. Since $G$ is $3$-colorable, $x\not\sim y$. By Claim \ref{ngem}, we have that 
$|N(x)\cap \{x_1,x_2,x_3,x_4\}|\neq 4$ and $|N(y)\cap \{x_1,x_2,x_3,x_4\}|\neq 4$.
Suppose $N(x)\cap \{x_1,x_2,x_3,x_4\}=N(y)\cap \{x_1,x_2,x_3,x_4\}$.
Since $x\not\sim y$,  Claim~\ref{nft} implies that there exist $x'\in N(x)\setminus N(y)$ and $y'\in N(y)\setminus N(x)$. Then $x'\sim x_1$ or $y'\sim x_1$, otherwise $x'xx_1yy'$ is an induced $P_5$ or $C_5$. Without loss of generality, we assume $x'\sim x_1$. Since $G$ is $3$-colorable, $x'\not\sim x_2$. Then $x_1x'xx_2y$ forms an induced gem, a contradiction.
So $N(x)\cap \{x_1,x_2,x_3,x_4\}\neq N(y)\cap \{x_1,x_2,x_3,x_4\}$. By the symmetry of $x$ and $y$, it suffices to consider the following cases.
If $N(x)\cap \{x_1,x_2,x_3,x_4\}=\{x_1,x_2\}$, then $xx_1yx_3x_4$ is an induced $P_5$ when $N(y)\cap \{x_1,x_2,x_3,x_4\}=\{x_1,x_2,x_3\}$ and $xx_1yx_4x_3$ is an induced $P_5$ when  $N_G(y)\cap \{x_1,x_2,x_3,x_4\}=\{x_1,x_2,x_4\}$,  a contradiction. 
If  $N(x)\cap \{x_1,x_2,x_3,x_4\}=\{x_1,x_2,x_3\}$ and  $N(y)\cap \{x_1,x_2,x_3,x_4\}=\{x_1,x_2,x_4\}$, then $x_2yx_1xx_3$ is an induced gem, a contradiction. Hence $\mathcal{F}_G$ is an empty set.
\end{proof}

\begin{figure}[!htbp]
	\centering
 \begin{subfigure}[t]{.23\textwidth}
	\centering
	\begin{tikzpicture}[scale=1.0]   
		\draw [line width=0.8pt, black] (0,3) to (0,0);
		\draw [line width=0.8pt, black] (0,2) to (-1,2);
		\draw [line width=0.8pt, black] (0,3) to (-1,2);
		\draw [line width=0.8pt, black] (0,2) to (1,2);
		\draw [line width=0.8pt, black] (0,3) to (1,2);

		\draw [fill=black,line width=0.8pt] (0,0) node[right=0.5mm] {$x_4$} circle (2pt); 
		\draw [fill=black,line width=0.8pt] (0,1) node[right=0.5mm] {$x_3$} circle (2pt); 
		\draw [fill=black,line width=0.8pt] (0,2) circle (2pt);
            \node at (0.4,1.65) {$x_2$};
		\draw [fill=black,line width=0.8pt] (0,3) node[above=0.5mm] {$x_1$} circle (2pt); 
		\draw [fill=black,line width=0.8pt] (-1,2) node[below=0.5mm] {$x$} circle (2pt);
		\draw [fill=black,line width=0.8pt] (1,2) node[below=0.5mm] {$y$} circle (2pt);

	\end{tikzpicture}
	\caption{$[\{x,y\},x_1x_2x_3x_4]$}
	\label{fig:H}
\end{subfigure}
	\begin{subfigure}[t]{.2\textwidth}
		\centering
		\begin{tikzpicture}[scale=1.0]  
		\draw [line width=0.8pt, black] (0,3) to (0,0);
		\draw [line width=0.8pt, black] (0,2) to (-1,2);
		\draw [line width=0.8pt, black] (0,3) to (-1,2);
		\draw [line width=0.8pt, black] (0,0) to (-1,2);

		\draw [fill=black,line width=0.8pt] (0,0) node[below=0.5mm] {$u_4$} circle (2pt); 
		\draw [fill=black,line width=0.8pt] (0,1) node[right=0.5mm] {$u_3$} circle (2pt); 
		\draw [fill=black,line width=0.8pt] (0,2) node[right=0.5mm] {$u_2$} circle (2pt); 
		\draw [fill=black,line width=0.8pt] (0,3) node[right=0.5mm] {$u_1$} circle (2pt); 
	\draw [fill=black,line width=0.8pt] (-1,2) node[left=0.5mm] {$u$} circle (2pt);

	\end{tikzpicture}
	\caption{$\overline{P_5}$}
	\label{fig:case1}
\end{subfigure}
	\begin{subfigure}[t]{.2\textwidth}
		\centering
		\begin{tikzpicture}[scale=1.0]  
		\draw [line width=0.8pt, black] (0,3) to (0,0);
		\draw [line width=0.8pt, black] (0,2) to (-1,2);
		\draw [line width=0.8pt, black] (0,3) to (-1,2);
		\draw [line width=0.8pt, black] (0,0) to (1,1);
		\draw [line width=0.8pt, black] (0,1) to (1,1);
		\draw [line width=0.8pt, black] (0,0) to (-1,2);

		\draw [fill=black,line width=0.8pt] (0,0) node[below=0.5mm] {$u_4$} circle (2pt); 
		\draw [fill=black,line width=0.8pt] (0,1)  circle (2pt); 
        \node at (-0.29,1.3) {$u_3$};
		\draw [fill=black,line width=0.8pt] (0,2) node[right=0.5mm] {$u_2$} circle (2pt); 
		\draw [fill=black,line width=0.8pt] (0,3) node[right=0.5mm] {$u_1$} circle (2pt); 
		\draw [fill=black,line width=0.8pt] (-1,2) node[left=0.5mm] {$u$} circle (2pt);
		\draw [fill=black,line width=0.8pt] (1,1) node[right=0.5mm] {$v$} circle (2pt);

	\end{tikzpicture}
	\caption{$v\not\sim u_1$}
	\label{fig:case1}
\end{subfigure}
\begin{subfigure}[t]{.2\textwidth}
	\centering
	\begin{tikzpicture}[scale=1.0]   
		\draw [line width=0.8pt, black] (0,3) to (0,0);
		\draw [line width=0.8pt, black] (0,2) to (-1,2);
		\draw [line width=0.8pt, black] (0,3) to (-1,2);
		\draw [line width=0.8pt, black] (0,0) to (1,1);
		\draw [line width=0.8pt, black] (0,1) to (1,1);
		\draw [line width=0.8pt, black] (0,0) to (-1,2);
		\draw [line width=0.8pt, black] (1,1) to (0,3);

		\draw [fill=black,line width=0.8pt] (0,0) node[below=0.5mm] {$u_4$} circle (2pt); 
		\draw [fill=black,line width=0.8pt] (0,1) circle (2pt);
        \node at (-0.29,1.3) {$u_3$};
		\draw [fill=black,line width=0.8pt] (0,2)  circle (2pt); 
        \node at (-0.29,2.3) {$u_2$};
		\draw [fill=black,line width=0.8pt] (0,3) node[right=0.5mm] {$u_1$} circle (2pt); 
		\draw [fill=black,line width=0.8pt] (-1,2) node[left=0.5mm] {$u$} circle (2pt);
		\draw [fill=black,line width=0.8pt] (1,1) node[right=0.5mm] {$v$} circle (2pt);

	\end{tikzpicture}
	\caption{$v\sim u_1$}
	\label{fig:case2.}
\end{subfigure}

\caption{Some related graphs in the proof of Claim~\ref{nP5}.}
\label{2cases}
\end{figure}
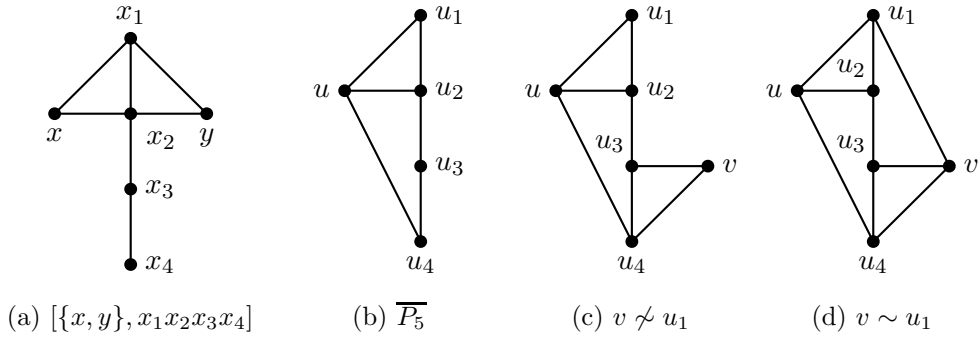

Now suppose that $G$ has an induced $\overline{P_5}$ shown in Figure~\ref{2cases}(b), where $u_1u_2u_3u_4$ is an induced $P_4$ and $u$ is adjacent to $u_1,u_2,u_4$. Since $u_4\not \sim u_2$, Claim~\ref{nft} implies that there exists a vertex  $v\in N(u_4)\setminus N(u_2)$. Note that $v\sim u_3$, otherwise $u_1u_2u_3u_4v$ is an induced  $P_5$ or $C_5$. We need to handle  the following two cases.

{\bf Case 1.} $v\not\sim u_1$. See Figure~\ref{2cases}(c).

Since $\delta(G)\geq 3$, there exists a vertex $x\in N(u_1)\setminus\{u,u_2\}$. By Subclaim~\ref{B}, $x\not \sim u_2$, otherwise $[\{u,x\}, u_1u_2u_3u_4]\in \mathcal{F}_G$. Then $x\sim u_3$, otherwise $xu_1u_2u_3u_4$ is an induced $P_5$ or $C_5$. It follows that  $x\not \sim u_4$, otherwise $[\{v,x\}, u_4u_3u_2u_1]\in \mathcal{F}_G$. Then $x\sim u$, otherwise $u_1xu_3u_4u$ is an induced $C_5$. Now $[\{u_2,x\}, u_1uu_4u_3]\in \mathcal{F}_G$, contrary to Subclaim~\ref{B}. 

{\bf Case 2.} $v\sim u_1$. See Figure~\ref{2cases}(d).

Note that $u\not\sim v$, otherwise  $uu_2u_1vu_4$ is an induced gem.
Then the graph shown in Figure~\ref{2cases}(d) is an induced graph of $G$, let $G_1$ denote the graph. If $G=G_1$, then do the same argument as Case 1 of Claim~\ref{nC5}, we have that $G$ is $k$-mixing, a contradiction. So $V(G)\setminus V(G_1)\neq \emptyset$.
Since all vertices of $G_1$ are symmetry, without loss of generality, we assume that $w\in V(G)\setminus V(G_1)$ and  $w\sim v$. 
Note that $uu_2u_3v$ is an induced $P_4$. It follows that $w\not\sim u_3$, otherwise $[\{w,u_4\}, vu_3u_2u]\in \mathcal{F}_G$. Then
$w\sim u_2$, otherwise $uu_2u_3vw$ is an induced $P_5$ or $C_5$.  Additionally, we have $w\not\sim u$ and $w\not\sim u_4$, otherwise $[\{w,u_1\}, uu_2u_3u_4]\in \mathcal{F}_G$ and $[\{w,u_3\}, u_4vu_1u]\in \mathcal{F}_G$.   Then we get that $uu_2wvu_4$ is an induced $C_5$, a contradiction. 
\end{proof} 
Now we get that $G$ is $3$-colorable $\{P_5,\overline{P_5},C_5\}$-free. Then by Lemma~\ref{3p5}, $G$ is $k$-mixing, contradicting our assumption. This completes the proof of Theorem
\ref{mainthm1}.
\qed

\section{Proof of Theorem~\ref{mainthm2}}\label{main2}

In this section, for any $t\geq 4$ and $t+1 \leq k \leq {t\choose2}$, we  construct  a $P_5$-free graph with chromatic number $t$ that has a frozen $k$-coloring. 

Let $S=\{u_1,u_2,\ldots,u_{2k}\}$. Let $T=\{\{a,b\}, a,b\in\{1,2,\ldots,t\}, a\neq b\}$, which has ${t\choose2}$ elements. Let $P=\{\{1,2\},\{2,3\},\ldots,\{t-1,t\},\{t,1\}\}$. Note that $P\subseteq T$. Let $\phi:S \to\{1,2,\ldots,k\}$ be a mapping such that $\phi(u_{2i-1})=\phi(u_{2i})=i$ for each $i\in\{1,2,\ldots,k\}$. Let $\alpha:S\to \{1,2,\ldots, t\}$ be a mapping such that $\{\alpha(u_{2i-1}),\alpha(u_{2i})\}\in P$ for $i\in\{1,2,\ldots,t\}$, $\{\alpha(u_{2i-1}),\alpha(u_{2i})\}\in T\setminus P$ for $i\in\{t+1,\ldots,k\}$, and $\{\alpha(u_{2i-1}),\alpha(u_{2i})\}\neq\{\alpha(u_{2j-1}),\alpha(u_{2j})\}$ for any $i, j\in\{1,\ldots,k\}$ and $i\neq j$. Note that it is possible because $t+1\leq k \leq\binom{t}{2}$. Now we construct a graph $G_{t,k}$ with the vertex set $S$ and any two vertices $u_i,u_j\in S$ is an edge of $G$ if and only if $\phi(u_i)\neq \phi(u_j)$ and $\alpha(u_i)\neq \alpha(u_j)$. Note that $\phi$ is a $k$-coloring and $\alpha$ is a $t$-coloring of $G_{t,k}$. So $G_{t,k}$ is $t$-colorable. Since $G_{t,k}[\{u_1,u_3,\dots,u_{2t-1}\}]$ is a clique with size $t$,  we have $\chi(G_{t,k})=t$. Next we show that  $G_{t,k}$  is $P_5$-free and $\phi$ is a frozen coloring of $G_{t,k}$. 
The graph $G_{4,5}$ shown in Figure~\ref{fig:4coloring5frozen}, where 
the label on $u_i$ represents the color pair $(\phi(u_i),\alpha(u_i))$ for $i\in\{1,2,\ldots, 10\}$.


\begin{figure}[htbp]
	\centering
	\begin{tikzpicture}[scale=1.2]   
		\draw [line width=0.8pt, black] (0,4) to (1,1);
  \draw [line width=0.8pt, black] (0,4) to (1.5,2);
  \draw [line width=0.8pt, black] (0,4) to (0,0);
		\draw [line width=0.8pt, black] (0,4) to (-1,0);
		\draw [line width=0.8pt, black] (0,4) to (-2,1);
		\draw [line width=0.8pt, black] (0,4) to (-1,4);
		\draw [line width=0.8pt, black] (1,3) to (1,1);
		\draw [line width=0.8pt, black] (1,3) to (0,0);
		\draw [line width=0.8pt, black] (1,3) to (-1,0);
  \draw [line width=0.8pt, black] (1,3) to (-2,1);
  \draw [line width=0.8pt, black] (1,3) to (-2,3);
		\draw [line width=0.8pt, black] (1,3) to (-2.5,2);
        \draw [line width=0.8pt, black] (1,3) to (-1,4);
        \draw [line width=0.8pt, black] (1.5,2) to (0,0);
        \draw [line width=0.8pt, black] (1.5,2) to (-1,0);
        \draw [line width=0.8pt, black] (1.5,2) to (-2,1);
        \draw [line width=0.8pt, black] (1.5,2) to (-2.5,2);
        \draw [line width=0.8pt, black] (1.5,2) to (-2,3);
        \draw [line width=0.8pt, black] (1.5,2) to (-1,4);
        \draw [line width=0.8pt, black] (1,1) to (-1,0);
        \draw [line width=0.8pt, black] (1,1) to (-2.5,2);
        \draw [line width=0.8pt, black] (1,1) to (-2,1);
        \draw [line width=0.8pt, black] (1,1) to (-2,3);
        \draw [line width=0.8pt, black] (0,0) to (-2,1);
        \draw [line width=0.8pt, black] (0,0) to (-2.5,2);
        \draw [line width=0.8pt, black] (0,0) to (-2,3);
        \draw [line width=0.8pt, black] (-1,0) to (-2.5,2);
        \draw [line width=0.8pt, black] (-1,0) to (-1,4);
        \draw [line width=0.8pt, black] (-1,0) to (-2,3);
        \draw [line width=0.8pt, black] (-2,1) to (-1,4);
        \draw [line width=0.8pt, black] (-2,1) to (-2,3);
        \draw [line width=0.8pt, black] (-2.5,2) to (-1,4);
		
		\draw [fill=black,line width=0.8pt] (0,4) circle (2pt);
        \node at (0,4.2) {$u_1$};
        \node at (0.5,4) {(\textcolor{black}{$1$},$1$)};
		\draw [fill=black,line width=0.8pt] (1,3) circle (2pt); 
        \node at (1,3.2) {$u_2$};
        \node at (1.5,3) {(\textcolor{black}{$1$},$2$)};
		\draw [fill=black,line width=0.8pt] (1.5,2) circle (2pt);  
        \node at (1.6,2.25) {$u_3$};
        \node at (2.0,2) {(\textcolor{black}{$2$},$2$)};
        \draw [fill=black,line width=0.8pt] (1,1) circle (2pt);
        \node at (1.3,1) {$u_4$};
        \node at (1.2,0.7) {(\textcolor{black}{$2$},$3$)};
        \draw [fill=black,line width=0.8pt] (0,0) circle (2pt);
        \node at (0.3,0) {$u_5$};
        \node at (0,-0.3) {(\textcolor{black}{$3$},$3$)};
        \draw [fill=black,line width=0.8pt]  (-1,0) circle (2pt);
        \node at (-1.3,0) {$u_6$};
        \node at (-1.2,-0.3) {(\textcolor{black}{$3$},$4$)};
        \draw [fill=black,line width=0.8pt] (-2,1) circle (2pt);
        \node at (-2.3,1) {$u_7$};
        \node at (-2.3,0.7) {(\textcolor{black}{$4$},$4$)};
        \draw [fill=black,line width=0.8pt] (-2.5,2) circle (2pt);
        \node at (-2.6,2.25) {$u_8$};
        \node at (-3.0,2) {({$4$},1)};
        \draw [fill=black,line width=0.8pt] (-2,3) circle (2pt);
        \node at (-1.9,3.2) {$u_9$};
        \node at (-2.5,3) {(\textcolor{black}{$5$},$1$)};
        \draw [fill=black,line width=0.8pt]  (-1,4) circle (2pt);
        \node at (-1,4.2) {$u_{10}$};
        \node at (-1.5,4) {(\textcolor{black}{$5$},$3$)};
        
	\end{tikzpicture}
	\caption{The graph $G_{4,5}$.}
	\label{fig:4coloring5frozen}
\end{figure}
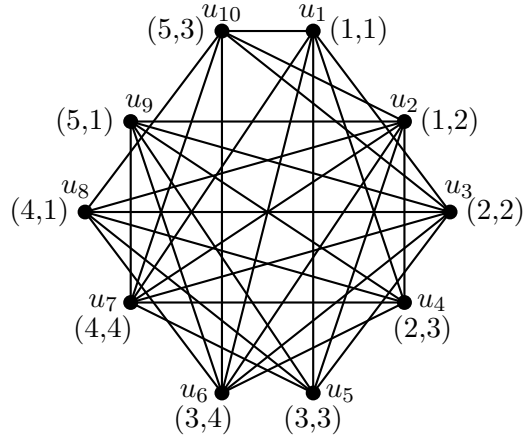

\begin{claim}
$\phi$ is a frozen $k$-coloring of $G_{t,k}$.    
\end{claim}
\begin{proof}
We need to prove $\{\phi(v): v\in N(u)\}=\{1,2,\ldots,k\}\setminus\{\phi(u)\}$ for $u\in S$. Without loss of generality, we prove that it holds for $u_1$.  In other words, we need to prove for any $2\leq j\leq k$, $u_1$  is adjacent to at least one vertex in $\{u_{2j-1}, u_{2j}\}$. Let $j\in\{2,\ldots,k\}$. Note that $\phi(u_1)\neq \phi(u_{2j-1})$ and $\phi(u_1)\neq \phi(u_{2j})$. Then by the construction of $G_{t,k}$, it suffices to prove  $\alpha(u_1)\neq \alpha(u_{2j-1})$ or $\alpha(u_1)\neq \alpha(u_{2j})$. This is established because  $\alpha(u_{2j-1})\neq \alpha(u_{2j})$.
\end{proof}

\begin{claim}
$G_{t,k}$  is $P_5$-free.   
\end{claim}
\begin{proof}
Suppose that $G_{t,k}$ contains an induced $P_5$, then we denote the induced $P_5$ by $P=x_1x_2x_3x_4x_5$. Let $\theta$  denote the mapping $\phi$  or $\alpha$. 

We first claim that for any two adjacent vertices $x_i$ and $x_j$ of $P$, if $x_k\not\sim x_i$ and $x_k\not\sim x_j$, then $\theta(x_k)=\theta(x_i)$ or $\theta(x_k)=\theta(x_j)$, where $i,j,k\in\{1,2,\ldots,5\}$. 
Suppose that $\theta=\phi$ and $\phi(x_k)\neq \phi(x_i)$,  $\phi(x_k)\neq \phi(x_j)$. Since $\alpha(x_i)\neq \alpha(x_j)$, $\alpha(x_k)\neq\alpha(x_i)$ or $\alpha(x_k)\neq \alpha(x_j)$. Then $x_k\sim x_i$ when $\alpha(x_k)\neq\alpha(x_i)$ and $x_k\sim x_j$ when $\alpha(x_k)\neq\alpha(x_j)$, a contradiction. Hence, $\phi(x_k)=\phi(x_i)$ or $\phi(x_k)=\phi(x_j)$. Similarly, we can prove that it holds for $\alpha$.

So we can obtain that $\theta(x_4)\in \{\theta(x_1), \theta(x_2)\}$ and  $\theta(x_5)\in \{\theta(x_1), \theta(x_2)\}\cap \{\theta(x_2), \theta(x_3)\}$. Note that $\theta(x_5)=\theta(x_2)$ when $\theta(x_1), \theta(x_2), \theta(x_3)$ are distinct. Let $\theta(x_1)=a$ and $\theta(x_2)=b$, where $a\neq b$. If $\theta(x_3)=a$, then $\theta(x_4)=b$ and $\theta(x_5)=a$. If $\theta(x_3)\neq a$, then let $\theta(x_3)=c$, where $a,b,c$ are pairwise distinct.  It follows that  $\theta(x_5)=b$ and   $\theta(x_4)=a$. So $P$ has only two types of coloring $ababa$ or $abcab$ under $\theta$. Since each color appears exactly twice under $\phi$, the type of coloring of $P$ can only be $abcab$  under $\phi$. Note that $\alpha(x_i)\neq \alpha(x_j)$ when $\phi(x_i)= \phi(x_j)$. So  the types of coloring of $P$ under $\phi$ and $\alpha$ are distinct. Hence, the type of coloring of $P$ can only be $ababa$  under $\alpha$. Note that $\phi(x_1)= \phi(x_4)$, $\phi(x_2)= \phi(x_5)$ and $\phi(x_1)\neq \phi(x_2)$. By the definitions of $\phi$ and $\alpha$, we have $\{\alpha(x_1),\alpha(x_4)\}\neq \{\alpha(x_2),\alpha(x_5)\}$, a contradiction. Therefore, $G_{t,k}$  is $P_5$-free.
\end{proof}

\section{Conclusion}\label{main3}

\noindent Combining Theorems \ref{mainthm1} and \ref{mainthm2}, the connectivity of $ R _k(G) $ concerning  $t$-chromatic $P_5$-free graphs $G$ is still unclear for some values of $k$ and $t$, so we propose the following problem.


\begin{problem}\label{Problem-larger-tchoose2}
For any $t\geq4$ and $k\geq{t\choose2}+1$, does there exist a  $t$-chromatic $ P_5 $-free graph $G$ such that $\mathcal{R}_{k}(G)$ is disconnected?
\end{problem}

Proposition \ref{p6free} might be useful for providing an affirmative answer to Problem \ref{Problem-larger-tchoose2}. Specifically, if there exists a $t$-chromatic $ P_5 $-free graph $G$ with $\omega(G)=t$ that has a frozen $k$-coloring, where $k\geq t+1\geq5$, then for any $s\geq 1$, there exists a $(t+s)$-chromatic $ P_5 $-free graph $G$ with $\omega(G)=t+s$ that has a frozen $(k+s)$-coloring. 
Consequently, if such graphs exist for $t=4$ with each $k\geq t+1$, then Problem \ref{Problem-larger-tchoose2} has an affirmative answer. 
It is worth noting that, we can not apply Proposition \ref{p6free} to improve the upper bound of $k$ in Theorem \ref{mainthm2}, since ${t-1\choose2}+1\leq {t\choose2}$ for each $t\geq4$.

\subsection*{Acknowledgements}

\noindent Hui Lei was partially supported by the National
Natural Science Foundation of China (No. 12371351) and the Young Elite Scientist Sponsorship Program by CAST. Yulai Ma was
partially supported by Sino-German (CSC-DAAD) Postdoc Scholarship Program 2021
(57575640), and Deutsche Forschungsgemeinschaft(DFG, German Research Foundation)-445863039. Zhengke Miao was partially supported by the National Natural Science
Foundation of China (No. 12031018). Yongtang Shi and Susu Wang were partially supported by the National Natural Science Foundation
of China (No. 12161141006), the Natural Science Foundation of Tianjin (No.
20JCJQJC00090) and the Fundamental Research Funds for the Central Universities,
Nankai University.

\end{document}